\documentclass{article}
\usepackage[english]{babel} 		
\usepackage{fullpage}				
\usepackage[utf8]{inputenc}			
\usepackage{amsmath, amssymb,amsthm}
\usepackage{graphicx}				
\usepackage{hyperref}				
\usepackage{xcolor}
\usepackage{multirow}
\usepackage{tikz}
\usepackage{ytableau}
\usepackage{mathrsfs}
\usepackage{thmtools}
\usepackage{thm-restate}
\usepackage{varwidth}
\usepackage{comment}
\usepackage{mathtools}
\usepackage[mathscr]{euscript}
\usepackage{bbm}
\usepackage{multicol}
\usepackage{enumerate}
\usepackage{graphicx} 
\theoremstyle{definition}
\newtheorem{thm}{Theorem}[section]
\newtheorem{lem}[thm]{Lemma}
\newtheorem*{lem*}{Lemma}
\newtheorem*{thm*}{Theorem}
\newtheorem{prop}[thm]{Proposition}

\newtheorem{cor}[thm]{Corollary}
\newtheorem{conj}[thm]{Conjecture}
\newtheorem{defn}[thm]{Definition}
\newtheorem*{remark*}{Remark}
\newtheorem{remark}{Remark}
\newtheorem{example}{Example}

\newtheorem{cor/defn}[thm]{Corollary/Definition}
\DeclareMathOperator{\Par}{\mathbb{Y}}
\DeclareMathOperator{\fX}{\mathfrak{X}}
\DeclareMathOperator{\fY}{\mathfrak{Y}}
\DeclareMathOperator{\fp}{\mathfrak{p}}
\DeclareMathOperator{\fq}{\mathfrak{q}}
\DeclareMathOperator{\sO}{\mathscr{O}}
\DeclareMathOperator{\cJ}{\mathcal{J}}
\DeclareMathOperator{\cA}{\mathcal{A}}
\DeclareMathOperator{\bP}{\mathbb{P}}
\DeclareMathOperator{\SSYT}{\mathrm{SSYT}}
\DeclareMathOperator{\Exp}{\mathrm{Exp}}

\DeclareMathOperator{\Gal}{\mathrm{Gal}}
\DeclareMathOperator{\GL}{\mathrm{GL}}
\DeclareMathOperator{\Tr}{\mathrm{Tr}}

\DeclareMathOperator{\bA}{\mathbb{A}}
\DeclareMathOperator{\sL}{\mathscr{L}}
\DeclareMathOperator{\Id}{\mathrm{Id}}
\DeclareMathOperator{\Ind}{\mathrm{Ind}}
\DeclareMathOperator{\type}{\mathrm{type}}
\DeclareMathOperator{\sh}{\mathrm{sh}}
\DeclareMathOperator{\Sym}{\mathrm{Sym}}

\usepackage[
    backend=bibtex,
    maxnames=5,
    style=alphabetic
  ]{biblatex}
\title{Artin Symmetric Functions}
\author{Milo Bechtloff Weising}
\date{\today}

\begin{document}

\maketitle

\abstract{In this paper we construct an algebraic invariant attached to Galois representations over number fields. This invariant, which we call an Artin symmetric function, lives in a certain ring we introduce called the ring of arithmetic symmetric functions. This ring is built from a family of symmetric functions rings indexed by prime ideals of the base field. We prove many necessary basic results for the ring of arithmetic symmetric functions as well as introduce the analogues of some standard number-theoretic objects in this setting. We prove that the Artin symmetric functions satisfy the same algebraic properties that the Artin L-functions do with respect to induction, inflation, and direct summation of representations. The expansion coefficients of these symmetric functions in different natural bases are shown to be character values of representations of a compact group related to the original Galois group. In the most interesting case, the expansion coefficients into a specialized Hall-Littlewood basis come from new representations built from the original Galois representation using polynomial functors corresponding to modified Hall-Littlewood polynomials. Using a special case of the Satake isomorphism in type GL, as formulated by Macdonald, we show that the Artin symmetric functions yield families of functions in the (finite) global spherical Hecke algebras in type GL which exhibit natural stability properties. We compute the Mellin transforms of these functions and relate them to infinite products of shifted Artin L-functions. We then prove some analytic properties of these Dirichlet series and give an explicit expansion of these series using the Hall-Littlewood polynomial functors.}
\tableofcontents

\section{Introduction}

The \textit{\textbf{Artin L-functions}}, $L(s,\rho;L/K)$, were first introduced by Artin \cite{Artin_23} as invariants attached to finite dimensional representations $\rho: \Gal(L/K) \rightarrow \GL(V)$ of a finite Galois extension of number fields $L/K.$ One of the main motivations for the introduction of the functions $L(s,\rho;L/K)$ was to provide a factorization of the \textit{\textbf{Dedekind zeta function}} $\zeta_{L}(s)$ for an arbitrary Galois extension of number fields. The resulting factorization

\begin{equation}\label{Artin product for L}
    \zeta_{L}(s) = \zeta_{K}(s) \prod_{\rho \neq \mathbbm{1}} L(s,\rho;L/K)^{\dim(\rho)}
\end{equation}
 includes representation/number-theoretic data about each non-trivial irreducible representation 
$\rho$ of the Galois group $\Gal(L/K).$ The functions $L(s,\rho;L/K)$ are \textit{defined} as certain \textbf{\textit{Euler products}} with factors coming from the characteristic polynomials of $\rho(\sigma_{\fq})$ where $\sigma_{\fq}$ is any choice of \textbf{\textit{Frobenius element}} at the prime $\fq$. This element is only well-defined up to conjugacy and special care must be taken when the prime $\fq$ of $K$ is \textbf{\textit{ramified}} in L as in this case there is not a well-defined conjugacy class of $\Gal(L/K)$ which lifts the Frobenius at $\fp|\fq.$ However, as Artin shows, this problem is alleviated by passing through the invariants $V^{I(\fq)}$ of $V$ for an \textbf{\textit{inertial subgroup}} $I(\fq)$ of $\fq.$ Artin shows that these functions $L(s,\rho;L/K)$ satisfy many exceptional algebraic properties relating to \textbf{\textit{induction/inflation}} and direct sum operations on the representations $\rho.$ These ensure in particular, that any continuous complex representation $\rho: \Gal(\overline{K}/K) \rightarrow \GL(V)$ yields a well-defined function $L(s;\rho;L/K)$ where we may choose $L:= \mathrm{Fix}(\mathrm{Ker}(\rho)).$

For an abelian extension $L/K$ the situation becomes simpler because all of the irreducible representations $\rho$ of $\Gal(L/K)$ are $1$-dimensional corresponding to group characters. In this case, the Artin L-functions align with the Dirichlet L-functions $L(s,\chi)$ which enjoy many exceptional analytic properties. Using now-standard techniques (Mellin transforms, Poisson summation) one may show that Dirichlet L-functions define meromorphic complex analytic functions with argument $s$ which satisfy functional equations similar to the famous functional equation of the (completed) Riemann zeta function: $\zeta^{*}(s)=\zeta^{*}(1-s).$ The abelian case was dealt with by Weber and Hecke (See Cogdell's paper \cite{Cogdell_07} for a more complete history) before Artin's work extended to the non-abelian situation. In the non-abelian case, the $L(s,\rho;L/K)$ in general satisfy a functional equation defining meromorphic functions on $\mathbb{C}$ but the constants appearing in these functional equations remain somewhat mysterious in many cases. Importantly, for any non-trivial abelian character $\chi \neq 1$, $L(s,\chi)$ defines, by analytic continuation, an \textbf{\textit{entire function}} on $\mathbb{C}$. That is to say $L(s,\chi)$ has no poles. In light of \eqref{Artin product for L}, this shows that for any \textit{abelian} extension $L/K$ of number fields the ratio of Dedekind zeta functions $\zeta_{L}(s)/\zeta_{K}(s)$ defines an entire function by analytic continuation. As a direct generalization Artin made the following conjecture:

\begin{conj}[Artin's Conjecture]\label{Artin's Conjecture}
    For any \textbf{non-trivial} irreducible Galois representation $\rho$ the function $L(s,\rho;L/K)$ may be extended by analytic continuation to an entire function on $\mathbb{C}.$
\end{conj}
An immediate consequence of Artin's conjecture would be that for any finite extension of number fields $L/K$, the ratio $\zeta_{L}(s)/\zeta_{K}(s)$ defines an entire function by analytic continuation. 

Artin's conjecture is still open in most cases. However, a lot of progress had been made towards Artin's conjecture through the \textbf{\textit{Langlands Program}} (for $\GL$) which roughly predicts a correspondence between Galois representations $\Gal(\overline{K}/K) \rightarrow \GL_n(\mathbb{C})$ and \textbf{\textit{automorphic forms}} $f: \GL_n(\mathbb{A}_{K}) \rightarrow \mathbb{C}$ such that the Artin L-function of $\rho$ equals the automorphic L-function of $f.$ One benefit of such a correspondence is that many analytic properties are easier to prove for automorphic L-functions. For $1$-dimensional Galois representations the Langlands correspondence exists and goes by the name \textbf{\textit{class field theory}}.

In this paper, we will introduce a family of invariants attached to Galois representations called the \textit{\textbf{Artin symmetric functions}} $\sL(\fX_{K},\rho;L/K).$ These are algebraic-combinatorial objects which live in the \textit{\textbf{ring of arithmetic symmetric functions}} associated to the base field $K.$ This ring is a graded completion $\mathcal{A}_K$ of a ring, which we write as $A_K$, defined as an infinite tensor product of rings of symmetric functions in variables $\fX^{(\fq)}$ labelled by prime ideals $\fq \in \mathbb{P}_{K}$ of the ring of integers $\mathcal{O}_{K}$. On its face, the rings $A_K \subset \mathcal{A}_K$ are not very interesting. However, as is common in algebraic combinatorics with the study of ordinary symmetric functions $\Lambda$, we find interesting structure inside the rings $A_K \subset \mathcal{A}_K$ by looking at special bases defined combinatorially. Specifically, the ring $A_K$, has a basis $\{P_{\lambda^{\bullet}}[\fX_{K}]\}_{\lambda^{\bullet}\in \mathbb{Y}_K}$ indexed by almost everywhere empty sequences of partitions $\lambda^{\bullet} = (\lambda^{(\fq)})_{\fq \in \mathbb{Y}_{K}}$ called the \textbf{\textit{arithmetic Hall-Littlewood functions}} which are defined as products of ordinary Hall-Littlewood $P$-functions $P_{\lambda^{(\fq)}}[\fX^{(\fq)};N(\fq)^{-1}].$ This construction is highly inspired by Chapter IV of Macdonald's book \cite{Macdonald} where a similar collection of functions is constructed in the setting of $ \mathbb{F}_{q}(t)$ to describe the characters of $\GL_{n}(\mathbb{F}_{q}).$ We develop a basic theory for $\mathcal{A}_K$ with respect to the basis $P_{\lambda^{\bullet}}[\fX_{K}]$ and relate this theory to the finite global spherical Hecke algebras $\mathcal{H}_{n}(K)$ for the locally compact groups $ \GL_n(\widehat{\mathcal{O}}_{K}) \backslash \GL_n(\mathbb{A}_{K}^{f})/ \GL_n(\widehat{\mathcal{O}}_{K})$ for $n \geq 1.$ 

This paper's main results are as follows. We begin the paper by developing a basic theory for the ring of arithmetic symmetric functions. We define and study objects in this theory which are analogous to objects in number theory including the Dedekind zeta function (Definition \ref{Dedekind sym defn}), the ideal norm map (Proposition \ref{norm map prop}), and the ideal class group (Propositions \ref{prop of class ring} and \ref{class ring is fin gen}). Next we turn our focus to the main object of this paper. We define the Artin symmetric functions in Definition \ref{main def for Artin sym} as infinite products similar to the original definition of the Artin L-functions. In Propositions \ref{Artin sym function properties} and \ref{Artin sym functions factorization} we prove that the Artin symmetric functions $\sL(\fX_{K},\rho;L/K) \in \mathcal{A}_K$ satisfy analogous representation-theoretic properties to the Artin L-functions $L(s,\rho;L/K).$ This includes induction and inflation. The proofs of these statements will follow mostly from the Artin's original proofs of these properties of $L(s,\rho;L/K).$ We compute the expansion of $\sL(\fX_{K},\rho;L/K)$ into the monomial and Schur function bases of $A_K$. Next, we compute the following expansion of $\sL(\fX_{K},\rho;L/K)$ into the arithmetic Hall-Littlewood basis:
\begin{equation}\label{HL expansion for intro}
    \sL(\fX_K,\rho;L/K) = \sum_{\lambda^{\bullet} \in \mathbb{Y}_K}  \Tr_{\widetilde{\mathbb{H}}_{\lambda^{\bullet}}(W_{\rho})}(\sigma_{K})  g(\lambda^{\bullet})^{-1} P_{\lambda^{\bullet}}[\mathfrak{X}_K]
\end{equation}
 Here the coefficients $\Tr_{\widetilde{\mathbb{H}}_{\lambda^{\bullet}}(W_{\rho})}(\sigma_{K})$ are trace values for a compact group $G_{L/K}$ for a special element $\sigma_{K} \in G_{L/K}$ and the representations $\widetilde{\mathbb{H}}_{\lambda^{\bullet}}(W_{\rho})$ are certain finite dimensional representations built from $\rho$ using what we call \textbf{\textit{modified Hall-Littlewood functors}}. These functors are essentially built by applying the classical Schur-Weyl duality to the modified Hall-Littlewood functions $\widetilde{H}_{\lambda}[\fX;t].$ We use the Lascoux-Sch\"{u}tzenberger combinatorial formula for the modified Kostka-Foulkes polynomials $\widetilde{K}_{\mu,\lambda}(t)$ \cite{LS_78} involving the \textbf{\textit{cocharge}} statistic on semi-standard Young tableaux to give the arithmetic Hall-Littlewood expansion coefficients in terms of simpler Schur functor trace values.
 
 In the last section of the paper, we use the correspondence Theorem \ref{extension of Macdonald map prop} to construct continuous $\GL_n(\widehat{\mathcal{O}}_K)$-biinvariant functions $F^{(n)}_{\rho,L/K}$ on $\GL_n(\mathbb{A}_{K}^{f})$ coming from the Artin symmetric functions $\sL(\fX_{K},\rho;L/K)$. We prove some elementary properties of the $F^{(n)}_{\rho,L/K}$ (Corollary \ref{properties of sph hecke elements}). From the functions $F^{(n)}_{\rho,L/K}$ we compute their corresponding Dirichlet series $L^{(n)}(s,\rho;L/K)$ via the \textbf{\textit{Mellin transform}} which we show in Theorem \ref{shifted product of L functions expansion} are in fact given by 

 \begin{equation}\label{shifted zeta intro}
     L^{(n)}(s,\rho;L/K) = \prod_{i=1}^{n} L(s-i+1,\rho;L/K) = \sum_{\lambda^{\bullet}\in \mathbb{Y}_{K}^{(n)}} \frac{\kappa^{(n)}(\lambda^{\bullet})}{||\lambda^{\bullet}||^{s}} \Tr_{\widetilde{\mathbb{H}}_{\lambda^{\bullet}}(W_{\rho})}(\sigma_{K}).
 \end{equation}
  Here the coefficients $\kappa^{(n)}(\lambda^{\bullet})$ do not depend on $\rho$ and $||\lambda^{\bullet}|| := \prod_{\fq \in \mathbb{P}_K}N(\fq)^{|\lambda^{(\fq)}|}.$ After some standard analytic arguments, we show that by taking a re-normalized limit of \eqref{shifted zeta intro} we obtain an equality (Theorem \ref{main thm 2}) of analytic functions in the half plane $\Re(s) > 1:$

  \begin{equation}\label{main thm 2 into}
     \widetilde{L}(s,\rho;L/K):= \prod_{j \geq 0} L(s+j, \rho ;L/K) = \sum_{\lambda^{\bullet} \in \Par_K} \frac{\gamma(\lambda^{\bullet})}{||\lambda^{\bullet}||^{s}} \Tr_{\widetilde{\mathbb{H}}_{\lambda^{\bullet}}(W_{\rho})}(\sigma_{K}).
  \end{equation}

Here the coefficients $0 < \gamma(\lambda^{\bullet}) \leq 1$ are explicit rational numbers not depending on $\rho.$ The analytic functions $\widetilde{L}(s,\rho;L/K)$ on the half plane $\Re(s)> 1$ satisfy the functional equations 
$$L(s,\rho;L/K)\widetilde{L}(s+1,\rho;L/K) = \widetilde{L}(s,\rho;L/K)$$ similar to the functional equation for the reciprocal gamma function: $\frac{s}{\Gamma(s+1)} = \frac{1}{\Gamma(s)}.$ The $\widetilde{L}(s,\rho;L/K)$ generalize the higher zeta functions of Momotani \cite{momotani_06} associated to Dirichlet characters. Following this terminology we may refer to the $\widetilde{L}(s,\rho;L/K)$ as higher Artin L-functions.

Lastly, in Theorem \ref{Analytic Continuation of Shifted Artin L-Function Product} we use these functional equations and the known analytic continuations of $L(s, \rho ;L/K)$ to show that the functions 
$\widetilde{L}(s,\rho;L/K)$ may be extended to a meromorphic functions on $\mathbb{C}.$ Under the assumption that Artin's conjecture is true, we conclude that if $\rho$ is nontrivial and irreducible then $\widetilde{L}(s,\rho;L/K)$ is an entire function on $\mathbb{C}.$ A more complete picture for the Artin symmetric functions would include data about the infinite places of $K$ however, we will leave these consideration to later work.

\subsection{Acknowledgments}
The author would like to thank Daniel Orr and Monica Vazirani for helpful discussions about Hall-Littlewood polynomials and Hecke algebras. The author would also like to thank Sebastian Bozlee and Leo Herr for interesting conversations about Picard groups and for pointing out an error in earlier versions of Propositions \ref{prop of class ring} \ref{class ring is fin gen} as well as thank Gergely Harcos for a correction to the introduction of a previous version of this paper.

\section{Ring of Arithmetic Symmetric Functions}

\subsection{Symmetric Functions}

We start with a review of the basics of symmetric functions theory.

\begin{defn}
    A \textbf{\textit{partition}} is a (possibly empty) sequence $\lambda = (\lambda_1,\ldots, \lambda_{\ell})$ of weakly decreasing non-zero non-negative integers $\lambda_1 \geq \ldots \geq \lambda_{\ell} > 0.$ We write $\emptyset$ for the unique empty partition. The \textbf{\textit{length}} of $\lambda$ is $\ell(\lambda):= \ell$ and the \textbf{\textit{size}} of $\lambda$ is $|\lambda|:= \lambda_1+\ldots + \lambda_{\ell}.$ Whenever $ \ell(\lambda) \leq n$ we may consider $\lambda \in \mathbb{Z}^n$ by identifying $\lambda$ with the vector $(\lambda_1,\ldots, \lambda_{\ell},0,\ldots,0) \in \mathbb{Z}^n.$ Let $m_{i}(\lambda)$ denotes the number of copies of $i$ occurring in $\lambda.$ We will write $\Par$ for the set of all partitions.  
\end{defn}

To match with the conventions for the rest of this paper we will consider symmetric functions with coefficients in $\mathbb{Q}^{\textrm{ab}}$. 

\begin{defn}
For $n \geq 1$ let $\mathfrak{S}_n$ denote the symmetric group of permutations on the set $\{1,\ldots, n\}.$ Define the ring of symmetric functions $\Lambda$ as the \textit{graded} inverse limit of the symmetric polynomial rings 
$$\Lambda \cong \varprojlim \mathbb{Q}^{\text{ab}}[x_1,\ldots, x_n]^{\mathfrak{S}_n}.$$ Given a partition $\lambda = (\lambda_1,\ldots,\lambda_r)$ we define:
\begin{itemize}
    \item $m_{\lambda} := \sum_{\alpha \sim \lambda } x^{\alpha}$ ranging over permutations $\alpha$ of $\lambda$
    \item $h_{\lambda}:= h_{\lambda_1}\cdots h_{\lambda_r}$ where 
    $h_m:= \sum_{i_1 \leq \ldots \leq i_m} x_{i_1}\cdots x_{i_m}$
    \item $p_{\lambda}:= p_{\lambda_1}\cdots p_{\lambda_r}$ where $p_{m}:= \sum_{i} x_i^{m}$
    \item $s_{\lambda}:= \sum_{T \in \SSYT(\lambda)} x^{T}$ where $\SSYT(\lambda)$ is the set of \textit{\textbf{semi-standard Young tableaux}} with shape $\lambda.$
\end{itemize}
The $m_{\lambda}$, $h_{\lambda},$ $p_{\lambda},$ and $s_{\lambda}$ are called the \textbf{\textit{monomial}}, \textbf{\textit{homogeneous}}, \textbf{\textit{power sum}} and \textbf{\textit{Schur}} symmetric functions respectively.

\end{defn}

\begin{remark}\label{Dedekind remark}
We will occasionally use \textbf{\textit{plethystic}} notation in this paper. We refer the reader to Macdonald's book \cite{Macdonald} for a proper introduction to plethysm. Roughly, however, we may describe the plethystic evaluation $F[g]$ of a symmetric function $F$ at some element $g = \sum_{i} c_i x^{\alpha_i}$ of a ring of formal power series with some coefficients $c_i \in \mathbb{Q}^{\mathrm{ab}}$ monomials $x^{\alpha_i}$ by first defining 
$$p_r[g]:= \sum_{i} c_i (x^{\alpha_i})^r$$ and then requiring that the map $F \rightarrow F[g]$ be a ring homomorphism. The \textbf{\textit{plethystic exponential}} is the special symmetric function given by 
$$\Exp[\fX]:= \prod_{i \geq 1}(1-x_i)^{-1} = \sum_{n \geq 0} h_n[\fX].$$ The plethystic exponential satisfies the important properties that $\Exp[0] = 1$ and $\Exp[\fX+ \fY] = \Exp[\fX]\Exp[\fY].$
\end{remark}

\begin{defn}
    The Hall inner product $\langle-,-\rangle_{H}$ on symmetric functions $\Lambda$ is defined as 
$$\langle h_{\mu}, m_{\lambda} \rangle = \delta_{\mu,\lambda}.$$
\end{defn}

The Hall inner product may also be characterized by the following.

\begin{prop}
    $$\langle p_{\lambda}, p_{\mu} \rangle_{H} := \delta_{\lambda,\mu} \prod_{i \geq 1}(m_{i}(\lambda) i^{m_{i}(\lambda)})$$
    $$\langle s_{\mu}, s_{\lambda} \rangle_{H} = \delta_{\mu, \lambda}$$
\end{prop}
\begin{proof}~
    [(4.7), pg. 64, \cite{Macdonald}] and [(4.8), pg. 64, \cite{Macdonald}].
\end{proof}

\begin{defn}
    For a partition $\lambda$ with $ \ell(\lambda) \leq n$ define 
    $$P_{\lambda}(x_1, \ldots, x_n;t):= \sum_{w \in \mathfrak{S}_n/\mathfrak{S}_{n}^{\lambda}} w \left( x_1^{\lambda_1}\cdots x_n^{\lambda_n} \prod_{\lambda_i > \lambda_j} \frac{x_i-tx_j}{x_i-x_j}\right) \in \mathbb{Q}[t][x_1,\ldots,x_n]^{\mathfrak{S}_n}$$ where $\mathfrak{S}_{n}^{\lambda}$ is the subgroup of $\mathfrak{S}_n$ stabilizing $\lambda \in \mathbb{Z}^n.$
\end{defn}

\begin{lem}
    For $\ell(\lambda) \leq n$ 
    $$P_{\lambda}(x_1,\ldots,x_n;t) = P_{\lambda}(x_1,\ldots,x_n, 0;t).$$
\end{lem}
\begin{proof}~
    [(2.5), pg. 208, \cite{Macdonald}]
\end{proof}

Due to the above stability, the following is well-defined:

\begin{defn}
    For a partition $\lambda$ define the \textbf{\textit{Hall-Littlewood symmetric function}} 
    $$P_{\lambda}[\fX;t] := \lim_{n} P_{\lambda}(x_1,\ldots,x_n;t) \in \mathbb{Z}[t]\otimes \Lambda.$$ Define the polynomial 
    $$b_{\lambda}(t):= \prod_{i \geq 1}(1-t)\cdots (1-t^{m_{i}(\lambda)}).$$ The \textbf{\textit{dual Hall-Littlewood symmetric functions}} are given by 
    $$Q_{\lambda}[\fX;t]:= b_{\lambda}(t) P_{\lambda}[\fX;t].$$
\end{defn}

The next result will be used repeatedly throughout this paper.

\begin{prop}\label{HL are basis in classic case}
    The functions $\{ P_{\lambda}[\fX;t] \}_{\lambda \in \Par}$ constitute a $\mathbb{Q}(t)$-basis for $\mathbb{Q}(t) \otimes \Lambda.$
\end{prop}
\begin{proof}~
    [(2.7), pg. 209, \cite{Macdonald}]
\end{proof}

The following bilinear form generalizes the Hall inner product and interacts well with the Hall-Littlewood symmetric functions.

\begin{defn}
    Define the non-degenerate bilinear pairing $\langle-,-\rangle_t$ on $\mathbb{Q}(t)\otimes \Lambda$ by 
    $$\langle p_{\mu}[\fX],p_{\lambda}[\fX]\rangle_t:= \delta_{\mu,\lambda} \prod_{i \geq 1}\left(\frac{m_{i}(\lambda) i^{m_{i}(\lambda)}}{1-t^{m_i(\lambda)}} \right).$$
\end{defn}

We require the next result later in this paper.

\begin{lem}[t-Cauchy Identity]\label{t cauchy identity}
    $$\Exp[(1-t)\fX\fY] = \sum_{\lambda} P_{\lambda}[\fX;t] Q_{\lambda}[\fY;t]$$
\end{lem}
\begin{proof}~
   [(4.4), pg. 224, \cite{Macdonald}]
\end{proof}

\subsection{Arithmetic Symmetric Functions}

We begin by introducing a generalization to the ordinary ring of symmetric functions $\Lambda$ in an infinite set of variables which has instead one infinite variable set for every prime ideal of the ring of integers of a number field. For the remainder of this paper $\mathbb{N}:= \{1,2,3,\ldots\}.$ We assume the reader has a familiarity with the basic objects in algebraic number theory.

\begin{defn}
    Let $K$ be a number field. We will use the following notation:
    \begin{itemize}
        \item $\sO_{K}:=$ ring of integers of $K$
        \item $\bP_K:=$  set of prime ideals of $K$
        \item $\cJ_K:=$ group of fractional ideals of $\sO_{K}$
        \item $\cJ_K^{+}:=$ monoid of integral ideals of $\sO_{K}$
        \item $\mathrm{Cl}_K:= \cJ_K/(K^{*}\mathcal{O}_K)$ is the ideal class group of $K$ 
        \item $K_{\fq}$ is the $\fq$-adic completion of $K$ at the prime ideal $\fq \in \mathbb{P}_K$
        \item $\sO_{K_{\fq}}:=$ the ring of integers of $K_{\fq}$
        \item $\bA_{K}^{f}:=$ the ring of finite adeles of $K$
        \item $\widehat{\sO}_K:= \prod_{\fq \in \bP_K} \sO_{K_{\fq}} \subset \bA_{K}^{f}.$ 
    \end{itemize}
\end{defn}

\begin{defn}\label{ring of arithmetic sym polys defn}
     For a number field $K$ define $\Par_{K}$ to be the set of all tuples $\lambda^{\bullet} = (\lambda^{(\fq)})_{\fq \in \bP_K}$ of partitions $\lambda^{(\fq)} \in \Par$ such that $\lambda^{(\fq)} = \emptyset$ for all but finitely many $\fq.$ For $\fq \in \bP_{K}$ define $\fX^{(\fq)}:= x_1^{(\fq)}+x_2^{(\fq)}+\ldots$ to be a set of commuting free variables. We will write $\fX_K:= (\fX^{(\fq)})_{\fq \in \bP_K}$ and assume that for distinct $\fq,\fp \in \bP_K$ the variable sets $\fX^{(\fq)}$ and $\fX^{(\fp)}$ mutually commute. Define $A_{K}$ to be the ring 
    $$A_{K}:= \bigotimes_{\fq \in \bP_K} \Lambda[\fX^{(\fq)}].$$ For the remainder of this paper we will always interpret the above infinite tensor product as 
    $$ \bigotimes_{\fq \in \bP_K}:= \lim_{\substack{S\subset \bP_K\\ |S| <\infty}} \bigotimes_{\fq \in S}$$ whenever it appears. We will write $F \in A_{K}$ as 
    $$F:= F[\fX_K].$$ We say that $F \in A_{K}$ is \textbf{\textit{homogeneous}} if $F$ is simultaneously homogeneous in each variable set $\fX^{(\fq)}$ for $\fq \in \bP_K.$ Given any homogeneous $F$ we define the norm of $F$ by
    $$||F||:= \prod_{\fq \in \bP_K} \fq^{\deg_{\fq}(F)} \in \cJ_K^{+}$$
    where $\deg_{\fq}(F)$ denotes the ordinary $\mathbb{Z}_{\geq 0}$ degree of $F$ in the variables $\fX^{(\fq)}.$ The map $F \rightarrow ||F||$ defines a $\cJ_{K}^{+}$-valued grading on $A_{K}$ where $\cJ_{K}^{+}$ is naturally ordered by divisibility. 
\end{defn}

\begin{defn}
    Let $\{u_{\lambda}\}_{\lambda \in \Par}$ be any of the standard homogeneous bases $m_{\lambda},h_{\lambda},s_{\lambda}$ for $\Lambda.$ For $\lambda^{\bullet} \in \Par_K$ define 
    $$u_{\lambda^{\bullet}}[\fX_K]:= \prod_{\fq \in \bP_K}u_{\lambda^{(\fq)}}[\fX^{(\fq)}] \in A_{K}.$$
\end{defn}

\begin{defn}\label{ring of arith sym functions defn}
    Given a number field $K$ we define the \textbf{\textit{ring of arithmetic symmetric functions of K}} as the completion of $A_{K}:$
    $$\cA_{K}:= \widehat{A_{K}}.$$ To be explicit, the ring $\cA_{K}$ consists of all formal sums of the form 
    $$\sum_{\mathfrak{a} \in \cJ_K^{+}}  F_{\mathfrak{a}}[\fX_K]$$ with $||F_{\mathfrak{a}}|| = \mathfrak{a}$ with multiplication given by 
    $$\left(\sum_{\mathfrak{a} \in \cJ_K^{+}}  F_{\mathfrak{a}}[\fX_K] \right)\left(\sum_{\mathfrak{b} \in \cJ_K^{+}}  G_{\mathfrak{b}}[\fX_K] \right) = \sum_{\mathfrak{d}\in \cJ_K^{+}} \left( \sum_{\mathfrak{a}\mathfrak{b} = \mathfrak{d}} F_{\mathfrak{a}}[\fX_K]G_{\mathfrak{b}}[\fX_K] \right).$$
\end{defn}

\begin{remark}
    For the majority of this paper we will work over $\mathbb{Q}^{\text{ab}}$ except for a few sections and we will make the change clear. The choice of $\mathbb{Q}^{\text{ab}}$ over, the possibly more natural choice of, $\mathbb{C}$ is to emphasize that the arithmetic symmetric functions are not analytic objects (on their own) and so we may take coefficients over any commutative ring. Most relevant for studying Galois representations, one could consider arithmetic symmetric functions over an $\ell$-adic field $E/ \mathbb{Q}_{\ell}$ and talk about $\ell$-adic Galois representations. We will not do so in this paper however.
\end{remark}

We may consider many constructions in the ring of arithmetic symmetric functions which are analogous to usual constructions in algebraic/analytic number theory. 

\begin{defn}\label{Dedekind sym defn}
    Define the \textbf{\textit{Dedekind symmetric function}} $\zeta_K[\fX_K] \in \cA_K$ as 
    $$\zeta_K[\fX_K]:= \prod_{\fq \in \bP_K}\prod_{i \geq 1}\left(1-x_i^{(\fq)} \right)^{-1}.$$ For a finite extension $L/K$ of number fields define the \textit{\textbf{relative Dedekind symmetric function}} $\zeta_{L/K}[\fX_K] \in \cA_K$ as 
    $$\zeta_{L/K}[\fX_K]:= \prod_{\fq \in \bP_K}\prod_{\substack{\fp \in \bP_L \\ \fp|\fq}}\prod_{i \geq 1}\left(1-(x_i^{(\fq)})^{f(\fp|\fq)} \right)^{-1}$$
    where $f(\fp|\fq):= \dim_{\sO_K/\fq}(\sO_L/\fp).$
\end{defn}

\begin{defn}
    Let $L/K$ be a finite extension of number fields. Define the \textbf{\textit{norm map}} $\mathcal{N}_{L/K}: \cA_{L} \rightarrow \cA_K$ by the substitution 
    $$\mathfrak{X}^{(\fp)} \rightarrow p_{f(\fp|\fq)}[\fX^{(\fq)}]$$
     whenever $\fp|\fq$ and extended as a graded algebra homomorphism to all of $\cA_L$ accordingly.
\end{defn}

The following proposition verifies that the norm map $\mathcal{N}_{L/K}$ is well-defined.

\begin{prop}
    If $F$ is homogeneous then 
    $||\mathcal{N}_{L/K}(F)|| = N_{L/K}(||F||)$ where $N_{L/K}: \cJ_{L}\rightarrow \cJ_{K}$ is the usual ideal norm map. 
\end{prop}
\begin{proof}
    \begin{align*}
        ||\mathcal{N}_{L/K}(F)|| &= \prod_{\fq \in \mathbb{P}_{K}} \fq^{\deg_{\fq}(\mathcal{N}_{L/K}(F))}\\
        &= \prod_{\fq \in \mathbb{P}_{K}} \fq^{\sum_{\fp|\fq} f(\fp|\fq) \deg_{\fp}(F)} \\
        &= \prod_{\fp \in \mathbb{P}_L} N_{L/K}(\fp)^{\deg_{\fp}(F)}\\
        &= N_{L/K}\left(\prod_{\fp \in \mathbb{P}_L} \fp^{\deg_{\fp}(F)}\right)\\
        &= N_{L/K}(||F||).\\
    \end{align*}
\end{proof}

\begin{remark}
Using plethystic notation 
$$\zeta_K[\mathfrak{X}_K] = \prod_{\fq \in \bP_K}\Exp[\fX^{(\fq)}] = \Exp \left[\sum_{\fq \in \bP_K} \fX^{(\fq)} \right]$$ and 
$$\zeta_{L/K}[\fX_K]= \prod_{\fq \in \bP_K} \Exp\left[\sum_{\substack{\fp \in \bP_L \\ \fp|\fq}} p_{f(\fp|\fq)}[\fX^{(\fq)}]  \right] = \Exp\left[ \sum_{\fq \in \bP_K}\sum_{\substack{\fp \in \bP_L \\ \fp|\fq}} p_{f(\fp|\fq)}[\fX^{(\fq)}]  \right].$$
Therefore, 
$$\zeta_{L/K}[\fX_{K}] = \mathcal{N}_{L/K}(\zeta_{L}[\fX_L]).$$
\end{remark}

Like the usual ideal norm map $N_{L/K}$, we show that $\mathcal{N}_{L/K}$ respects towers of extensions.

\begin{prop}\label{norm map prop}
    If $E/L/K$ is any tower of finite extensions of number fields then 
    $$\mathcal{N}_{E/K} = \mathcal{N}_{L/K}\mathcal{N}_{E/L}.$$
\end{prop}
\begin{proof}
    If $\fq_1|\fq_2|\fq_3$ then 
    $$f(\fq_1|\fq_3) = f(\fq_1|\fq_2)f(\fq_2|\fq_3).$$ Therefore, 
    $$\mathcal{N}_{E/K}(\fX^{(\fq_1)}) = p_{f(\fq_1|\fq_3)}[\fX^{(\fq_3)}] = p_{f(\fq_2|\fq_3)}[ p_{f(\fq_1|\fq_2)} [\fX^{(\fq_3)}] ] = \mathcal{N}_{L/K}\mathcal{N}_{E/L}(\fX^{(\fq_1)}).$$
\end{proof}

Using the above basic result regarding the norm map we show the following:

\begin{cor}
    If $E/L/K$ is any tower of finite extensions of number fields 
    $$\mathcal{N}_{L/K}(\zeta_{E/L}) = \zeta_{E/K}.$$
\end{cor}
\begin{proof}
$$\mathcal{N}_{L/K}(\zeta_{E/L}) = \mathcal{N}_{L/K}\mathcal{N}_{E/K}(\zeta_E) = \mathcal{N}_{E/K}(\zeta_E) = \zeta_{E/K}.$$
\end{proof}

Although we will not require the following results for the remainder of this paper, we include them to emphasize the analogy between classical objects in algebraic/analytic number theory and arithmetic symmetric functions. Readers interested only in the Artin symmetric functions may skip this part and go directly to subsection \ref{subsection 2.3}. 

\begin{defn}
Write $A_{K}(\mathbb{Z})$ and $\mathcal{A}_K(\mathbb{Z})$ for the rings defined as in Definitions \ref{ring of arithmetic sym polys defn} \ref{ring of arith sym functions defn} analogously but with coefficients in $\mathbb{Z}$ instead of $\mathbb{Q}^{\text{ab}}.$ We will similarly write $\Lambda_{\mathbb{Z}}$ for the ring of symmetric functions over $\mathbb{Z}.$
For $\mathfrak{a} \in \cJ_{K}$ define $\nu_{\fq}(\mathfrak{a}) \in \mathbb{Z}$ using unique factorization by 
  $$\mathfrak{a} = \prod_{\fq \in \mathbb{P}_K} \fq^{\nu_{\fq}(\mathfrak{a})}.$$ For $m \geq 0$ and $\mathfrak{a} \in \cJ_{K}$ define 
  $$H^{(m)}_{\mathfrak{a}}[\fX_K]:= h_{m}\left[ \sum_{\fq \in \mathbb{P}_K} \nu_{\fq}(\mathfrak{a}) \fX^{(\fq)}  \right] \in A_{K}(\mathbb{Z}).$$ Define 
  $$\Exp_{\mathfrak{a}}[\fX_K]:= \Exp\left[ \sum_{\fq \in \mathbb{P}_K} \nu_{\fq}(\mathfrak{a}) \fX^{(\fq)}  \right] = \sum_{m \geq 0} H^{(m)}_{\mathfrak{a}}[\fX_K] \in \mathcal{A}_K(\mathbb{Z}).$$
\end{defn}

Note that the $H^{(m)}_{\mathfrak{a}}[\fX_K]$ are \textbf{not} always homogeneous with respect to the grading $||\cdot||.$ However, if we define the simpler grading $\deg'$ on $A_K(\mathbb{Z})$ given by the degree in the $x_i^{(\fq)}$ variables (without distinguishing between distinct $\fq$), then the $H^{(m)}_{\mathfrak{a}}[\fX_K]$ \textbf{are} homogeneous with $\deg'(H^{(m)}_{\mathfrak{a}}[\fX_K]) = m.$ The following properties are straightforward to check.

\begin{lem}\label{exp lemma}~
    \begin{itemize}
        \item $\Exp_{(1)} = 1$
        \item $\Exp_{\mathfrak{a}\mathfrak{b}} = \Exp_{\mathfrak{a}} \Exp_{\mathfrak{b}} $
    \end{itemize}
\end{lem}

The following ring is an analogue of the ideal class group $\mathrm{Cl}_K$ of K in the setting of arithmetic symmetric functions. 

\begin{defn}
    Define the \textbf{\textit{class ring}} of K, $\mathscr{CL}_{K}$, as the quotient 
    $$\mathscr{CL}_{K}:= A_{K}(\mathbb{Z})/\mathscr{P}_{K}$$ 
    where 
    $$\mathscr{P}_{K}:= \langle H^{(m)}_{(\alpha)}[\fX_{K}] ~|~ m \geq 1, \alpha \in K^{*} \rangle.$$ Since $\mathscr{P}_K$ is a homogeneous ideal with respect to the grading $\deg'$ we may define the \textbf{\textit{completed class ring}} of K, $\widehat{\mathscr{CL}}_{K}$, as the $\deg'$-graded completion of the ring $\mathscr{CL}_{K}.$
\end{defn}

The following are some basic results regarding the class ring of number fields.
\pagebreak

\begin{prop}\label{prop of class ring}~
    \begin{itemize}
        \item The ring $\mathcal{O}_{K}$ is a unique factorization domain if and only if $\mathscr{CL}_{K} = \mathbb{Z}$.
        \item The map $\mathfrak{a} \rightarrow \Exp_{\mathfrak{a}}$ defines an injective group homomorphism $\mathrm{Cl}_K \rightarrow \widehat{\mathscr{CL}}_{K}^{*}.$
    \end{itemize}
\end{prop}
\begin{proof}
    For the first statement, suppose that $\mathcal{O}_K$ is a UFD. Then for each integral prime ideal $\fq \in \mathbb{P}_K$ and $ m\geq 1$, $h_{m}[\fX^{(\fq)}] = H^{(m)}_{\fq} \in \mathscr{P}_{K}.$ The set $\{h_{m}[\fX^{(\fq)}]~|~ m \geq 1, \fq \in \mathbb{P}_K \} $ generates $A_K(\mathbb{Z})$ algebraically so $\mathscr{CL}_{K} = A_K(\mathbb{Z})/\mathscr{P}_K = \mathbb{Z}.$ Conversely, if $\mathscr{CL}_{K} = \mathbb{Z}$ then by the algebraic independence of the homogeneous symmetric functions $h_{m}[\fX] \in \Lambda$ for $m \geq 1$, we know that $H^{(m)}_{\fq} = h_{m}[\fX^{(\fq)}] \in \mathscr{P}_K$ for all $m \geq 1$ and $\fq \in \mathbb{P}_K.$ Therefore, by the definition of $\mathscr{P}_K$ we must have in particular that $h_1[\fX^{(\fq)}]$ is an $A_K(\mathbb{Z})$ linear combination of elements of the form $H^{(m)}_{(\alpha)}$ for $m \geq 1$ and $\alpha \in K^*.$ By considering the degree of each $\fX^{(\fp)}$ variable set separately the only way this is possible is if $h_1[\fX^{(\fq)}] = H^{(1)}_{(\alpha)}$ for some $\alpha \in K^*$ which means that $\fq = (\alpha).$ Thus $\mathcal{O}_K$ is a UFD.

    For the second statement we first consider the map $E:\cJ_{K} \rightarrow \widehat{\mathscr{CL}}_{K}^{*}$ given by $E(\mathfrak{a}):= \Exp_{\mathfrak{a}}[\fX_K].$ This is a group homomorphism by Lemma \ref{exp lemma}. For any $\alpha \in K^*$ 
    \begin{align*}
        E(~(\alpha)~) &= \Exp_{(\alpha)}[\fX_K] \\
                      &= \Exp\left[ \sum_{\fq \in \mathbb{P}_K} \nu_{\fq}(~(\alpha)~) \fX^{(\fq)}  \right] \\
                      &= \sum_{m \geq 0} h_{m}\left[ \sum_{\fq \in \mathbb{P}_K} \nu_{\fq}(~(\alpha)~
                      ) \fX^{(\fq)}  \right] \\
                      &= \sum_{m \geq 0} H^{(m)}_{(\alpha)}\\
                      &= 1 \\
    \end{align*}
    and so $(\alpha) \in \mathrm{Ker}(E).$ Therefore, $E$ descends to a map $\mathrm{Cl}_K \rightarrow \widehat{\mathscr{CL}}_{K}^{*}.$ Suppose that $\mathfrak{a} \in \mathrm{Cl}_K$ is a \textbf{non-principal} ideal. Now suppose that for some $[\mathfrak{a}] \in \mathrm{Cl}_K,$ $E(\mathfrak{a}) = 1.$ Then for all $m \geq 1,$ 
    $$h_{m}\left[ \sum_{\fq \in \mathbb{P}_K} \nu_{\fq}(\mathfrak{a}) \fX^{(\fq)}  \right] \in \mathscr{P}_{K}.$$ In particular, for $m = 1$ we see that by considering the degree of  $h_{1}\left[ \sum_{\fq \in \mathbb{P}_K} \nu_{\fq}(\mathfrak{a}) \fX^{(\fq)}  \right] = \sum_{\fq \in \mathbb{P}_K} \nu_{\fq}(\mathfrak{a}) \fX^{(\fq)}$ in each variable set $\fX^{(\fq)}$ separately, that $\sum_{\fq \in \mathbb{P}_K} \nu_{\fq}(\mathfrak{a}) \fX^{(\fq)} = H^{(1)}_{(\alpha)}$ for some $\alpha \in K^{*}.$ Therefore, $\mathfrak{a} = (\alpha)$ and so $[\mathfrak{a}] = [(1)] \in \mathrm{Cl}_K.$ Thus the map $E:\mathrm{Cl}_K \rightarrow \widehat{\mathscr{CL}}_{K}^{*}$ is injective.
\end{proof}

It is an important result in classical algebraic number theory that the ideal class group $\mathrm{Cl}_K$ of any number field K is a \textbf{finite} group. We may show an analogue of this result for the ring $\mathscr{CL}_{K}$. Note that the rings $A_K(\mathbb{Z}), \mathcal{A}_K(\mathbb{Z})$ are $\lambda$-rings with the ordinary plethysm operators $h_r$ and extended accordingly. Equivalently, we may consider the power sum plethysm operators $p_r: x_i^{(\mathfrak{q})} \rightarrow (x_i^{(\mathfrak{q})})^r$ for all $ r \geq 1$ and expand each $h_r$ into the $p_{\lambda}$ accordingly. However, these expansions require coefficients in $\mathbb{Q}$ (for example $2h_2 = p_{1,1}+p_2$) but luckily the effect of computing $h_r[F]$ for any $F \in \mathcal{A}_K(\mathbb{Z})$ will still have integral coefficients.

\begin{prop}\label{class ring is fin gen}
    For any number field K the ring $\mathscr{CL}_{K}$ is a finitely generated $\lambda$-ring.
\end{prop}
\begin{proof}
    First, we argue that the $\lambda$-ring structure on $A_K(\mathbb{Z})$ descends to $\mathscr{CL}_{K}.$ Note that over $\mathbb{Z},$ the complete homogeneous symmetric functions $h_r$ for $ r \geq 1$ generate $\Lambda_{\mathbb{Z}}$ so it suffices to show that for $r \geq 1$ and any arbitrary element $F = \sum_{i=1}^{k} g_i H_{(\alpha_i)}^{(m_i)} \in \mathscr{P}_{K}$ we have that 
    $h_r[F] \in \mathscr{P}_{K}.$ To this end we have that 
    $$h_r[F] = h_r\left[ \sum_{i=1}^{k} g_i H_{(\alpha_i)}^{(m_i)} \right] = \sum_{a_1+\ldots + a_k = r} \prod_{i=1}^{k}h_{a_i}\left[ g_i H_{(\alpha_i)}^{(m_i)} \right] = \sum_{a_1+\ldots + a_k = r} \prod_{i=1}^{k} \left( \sum_{|\lambda| = a_i} s_{\lambda}[g_i]s_{\lambda}\left[ H_{(\alpha_i)}^{(m_i)} \right] \right).$$ By the Jacobi-Trudi identity, each $s_{\lambda}$ expands as a $\mathbb{Z}$-linear combination of the $h_{\mu}$ functions. Thus it suffices to show that for all $\ell \geq 1$ and $1\leq i \leq k$, $h_{\ell}\left[ H_{(\alpha_i)}^{(m_i)} \right] \in \mathscr{P}_{K}.$ We see that 
    $$h_{\ell}\left[ H_{(\alpha_i)}^{(m_i)} \right] = h_{\ell}\left[ h_{m_i}\left[\sum_{\mathfrak{q} \in \mathbb{P}_K} \nu_{\mathfrak{q}}(\alpha_i) \fX^{(\mathfrak{q})} \right] \right] = \left( h_{\ell}[h_{m_i}] \right) \left[\sum_{\mathfrak{q} \in \mathbb{P}_K} \nu_{\mathfrak{q}}(\alpha_i) \fX^{(\mathfrak{q})} \right] $$ which now certainly lives in the ideal $\mathscr{P}_{K}$. Therefore, $\mathscr{P}_K$ is a $\lambda$-ideal of $A_K(\mathbb{Z})$ and so $\mathscr{CL}_{K}$ is a $\lambda$-ring.

    Next we appeal to the classical result that $\mathrm{Cl}_K$ is a finite group. Let $\mathfrak{a}_1,\ldots, \mathfrak{a}_N$ be a finite set of representatives for the distinct (fractional) ideal classes in $\mathrm{Cl}_K.$ We claim that the elements $B_i:= \sum_{\fq \in \mathbb{P}_K} \nu_{\fq}(\mathfrak{a}_i) \fX^{(\fq)}$ generate $\mathscr{CL}_{K}$ as a $\lambda$-ring. Since $A_K(\mathbb{Z})$ is generated by the elements $h_r[\fX^{(\fq)}]$ for $\fq \in \mathbb{P}_K$, it suffices to show that $h_r[\fX^{(\fq)}]$ is contained in the $\lambda$-ring generated by $B_1,\ldots, B_N.$ Let $\fq \in \mathbb{P}_K$. By assumption $\fq = (\alpha) \mathfrak{a}_i$ for some $1\leq i \leq N$ and $\alpha \in K^{*}.$ Then 
    $$\fX^{(\fq)} = \sum_{\fp \in \mathbb{P}_K} \nu_{\fp}((\alpha) \mathfrak{a}_i) \fX^{(\fp)} =  \sum_{\fp \in \mathbb{P}_K} ( \nu_{\fp}((\alpha)) + \nu_{\fp}( \mathfrak{a}_i) ) \fX^{(\fp)} = \sum_{\fp \in \mathbb{P}_K}\nu_{\fp}((\alpha))\fX^{(\fp)} + \sum_{\fp \in \mathbb{P}_K} \nu_{\fp}( \mathfrak{a}_i)\fX^{(\fp)}.$$ Thus in the ring $\mathscr{CL}_{K}$
    $$\fX^{(\fq)} = \sum_{\fp \in \mathbb{P}_K} \nu_{\fp}( \mathfrak{a}_i)\fX^{(\fp)} = B_i.$$ But now we see that 
    $h_r[\fX^{(\fq)}] = h_r[B_i]$ in $\mathscr{CL}_{K}$ and so $h_r[\fX^{(\fq)}]$ is contained in the $\lambda$-ring generated by $B_1,\ldots, B_r$. Therefore, $\mathscr{CL}_{K}$ is finitely generated as a $\lambda$-ring.
\end{proof}

\subsection{Arithmetic Hall-Littlewood Functions}\label{subsection 2.3}
The most interesting basis we will be considering for the ring $A_K$ are built from the Hall-Littlewood symmetric functions $P_{\lambda}[\fX;t].$

\begin{defn}
    For $\lambda^{\bullet} \in \Par_{K}$ define $g(\lambda^{\bullet}) \in \mathbb{N}$ as 
    $$g(\lambda^{\bullet}):= \prod_{\fq \in \bP_K} N(\fq)^{n(\lambda^{(\fq)})}$$ where
    $$n(\lambda):= \sum_{i \geq 1}(i-1)\lambda_i.$$
\end{defn}

\begin{defn}
 For $\lambda^{\bullet} \in \Par_K$ define $P_{\lambda^{\bullet}}[\fX_K] \in A_K$ by 

$$P_{\lambda^{\bullet}}[\fX_K]:= \prod_{\mathfrak{q} \in \bP_{K}} P_{\lambda^{(\mathfrak{q})}}[\fX^{(\fq)};N(\mathfrak{q})^{-1}]$$
where $N(\mathfrak{a}):= |\sO_K/\mathfrak{a}|$ for $\mathfrak{a} \in \cJ_{K}^{+}.$ We will refer to the $P_{\lambda^{\bullet}}[\fX_K]$ as the \textbf{\textit{arithmetic Hall-Littlewood functions}} for $K.$ For $\lambda^{\bullet}\in \Par_K$ define 
$$b(\lambda^{\bullet}):= \prod_{\fq \in \bP_K} b_{\lambda^{(\fq)}}(N(\fq)^{-1}) = \prod_{\fq \in \bP_K}\prod_{i\geq 1} (1-N(\mathfrak{q})^{-1})\cdots (1-N(\mathfrak{q})^{-m_{i}(\lambda^{(\fq)})}).$$ We define the \textbf{\textit{dual Hall-Littlewood functions for}} $K$ as 
$$Q_{\lambda^{\bullet}}[\fX_K]:= b(\lambda^{\bullet})P_{\lambda^{\bullet}}[\fX_K].$$
\end{defn}

\begin{lem}
    The set $\{P_{\lambda^{\bullet}} \}_{\lambda^{\bullet} \in \Par_K}$  is a $\mathbb{Q}^{\text{ab}}$-basis for $A_K.$
\end{lem}
\begin{proof}
    For any parameter $t$ the set $\{P_{\lambda}[X;t] \}_{\lambda}$ is a basis for $\Lambda$ (Proposition \ref{HL are basis in classic case}). Therefore, for each $\fq \in \mathbb{P}_K$ the functions $\{P_{\lambda}[\fX^{(\fq)};N(\fq)^{-1}] \}_{\lambda}$ give a basis for $\Lambda[\fX^{(\fq)}]$ and so by taking products over all $\fq$ we obtain the result.
\end{proof}

We define an inner product on the space $A_K$ which is analogous to the $t$-Hall inner product $\langle -, - \rangle_t$ on $\mathbb{Q}(t) \otimes \Lambda.$

\begin{defn}
    Define the bilinear form $\langle-,-\rangle: A_K \times A_K \rightarrow \mathbb{Q}^{\text{ab}}$ by 
    $$\langle P_{\lambda^{\bullet}}, Q_{\mu^{\bullet}} \rangle = \delta_{\lambda^{\bullet}, \mu^{\bullet}}.$$ We will refer to $\langle-,-\rangle$ as the \textbf{\textit{arithmetic Hall inner product}} for $K.$
\end{defn}

Below are some basic properties of the pairing $\langle-,-\rangle.$

\begin{prop}
    The pairing $\langle-,-\rangle: A_K \times A_K \rightarrow \mathbb{Q}^{\text{ab}}$ is non-degenerate. Further, 
    $$\langle p_{\lambda^{\bullet}}, p_{\mu^{\bullet}} \rangle = z(\lambda^{\bullet}) \delta_{\lambda^{\bullet}, \mu^{\bullet}} $$
    where 
    $$z(\lambda^{\bullet}):= \prod_{\fq \in \bP_K} \prod_{i \geq 1} \left( \frac{m_i(\lambda^{(\fq)})!~i^{m_i(\lambda^{(\fq)})}}{1-N(\fq)^{-\lambda^{(\fq)}_i}} \right).$$
\end{prop}
\begin{proof}
We use [(4.11), pg.225, \cite{Macdonald}] applied to each local factor with $t= N(\fq)^{-1}$ and take a product over all prime ideals to obtain this result directly.
\end{proof}

Now we look at the analogue of the $t$-Cauchy Identity \ref{t cauchy identity} in the ring $\mathcal{A}_K.$

\begin{prop}
    $$ \prod_{\fq \in \bP_K}\prod_{i,j \geq 1} \left( \frac{1-N(\fq)^{-1}x_i^{(\fq)}y_j^{(\fq)}}{1-x_i^{(\fq)}y_j^{(\fq)}} \right) =  \sum_{\lambda^{\bullet} \in \Par_K} P_{\lambda^{\bullet}}[\fX_K]Q_{\lambda^{\bullet}}[\mathfrak{Y}_K]$$
\end{prop}
\begin{proof}
    This result follows by applying Lemma \ref{t cauchy identity} with  $t = N(\fq)^{-1}$ at each local part and then taking a product over all prime ideals. 
\end{proof}

\begin{remark}
   Although we will not emphasize this point throughout this paper, the algebra $A_K$ is isomorphic to the \textbf{\textit{Hall algebra}} of \textit{finite} $\widehat{\mathcal{O}}_K$ modules, i.e. those modules with a finite composition series. This is analogous to [(3.4), pg. 275, \cite{Macdonald}].
\end{remark}

\section{Artin Symmetric Functions}

\subsection{Definitions and Basic Properties}

In this section we introduce the Artin symmetric function attached to a given Galois representation over a number field. We will prove that they satisfy many of the same algebraic properties as the Artin L-functions. Using Brauer's theorem \ref{Brauer's thm} we obtain a factorization of the relative Dedekind symmetric functions $\zeta_{L/K}$ in terms of the Artin symmetric functions of irreducible representations of the Galois group of the extension.

\begin{defn}\label{main def for Artin sym}
    Suppose $L/K$ is a finite Galois extension of number fields. For each $\mathfrak{q}\in \bP_{K}$ arbitrarily pick $\mathfrak{p} \in \bP_{L}$ with $\mathfrak{p}|\mathfrak{q}$ i.e. $\mathfrak{p}\cap \sO_{K} = \mathfrak{q}.$ Let $I(\mathfrak{q}) \subset D(\mathfrak{q}) \subset \Gal(L/K)$ denote the inertia and decomposition groups of $\mathfrak{p}$ respectively. Let $\sigma_{\mathfrak{q}} \in D(\mathfrak{q})/I(\mathfrak{q})$ denote the corresponding choice of Frobenius element for $\mathfrak{q}.$ Suppose $\rho: \Gal(L/K) \rightarrow \GL(V)$ is a finite dimensional complex representation. Define the \textbf{\textit{Artin symmetric function}} for $\rho$ by 
    $$\sL(\fX_K,\rho;L/K):= \prod_{\mathfrak{q}\in \bP_K}\prod_{i \geq 1} \det(\Id_{V^{I(\mathfrak{q})}}- x_i^{(\mathfrak{q})}\rho(\sigma_{\mathfrak{q}})|_{V^{I(\mathfrak{q})}})^{-1} \in \cA_{K}.$$ Here $V^{I(\mathfrak{q})}$ denotes the subspace of $I(\mathfrak{q})$-fixed points of $V.$
\end{defn}

Note that the definition of $\sL(\fX_K,\rho;L/K)$ is independent of the choices of $\fp|\fq.$ In order to prove that the arithmetic symmetric functions $\sL(\fX_K,\rho;L/K)$ satisfy the same algebraic properties as the Artin L-functions we need to following notation.

\begin{defn}
    For $\fq \in \mathbb{P}_K$ define $f_{\fq}(x,\rho;L/K)$ to be the characteristic polynomial of $\rho(\sigma_{\mathfrak{q}})|_{V^{I(\mathfrak{q})}} $ in the variable x: 
    $$f_{\fq}(x,\rho;L/K):= \det(\Id_{V^{I(\mathfrak{q})}}- x\rho(\sigma_{\mathfrak{q}})|_{V^{I(\mathfrak{q})}}).$$
\end{defn}

\begin{lem}[Artin]\label{Artin's lemma}
    Suppose $E/L/K$ is a tower of finite Galois extensions,  $\rho,\gamma$ are two representations of $\Gal(L/K)$, and $\rho_0$ is a representation of $\Gal(E/L).$ Then for any $\fq \in \mathbb{P}_K$ the following properties hold:
    \begin{itemize}
        \item $$f_{\fq}(x,\rho;L/K)f_{\fq}(x,\gamma;L/K) = f_{\fq}(x,\rho \oplus \gamma;L/K)$$
        \item $$f_{\fq}(x,\rho;L/K) = f_{\fq}(x,\iota_{\Gal(L/K)}^{\Gal(E/K)}(\rho);E/K)$$ 
        \item $$ \prod_{\fp|\fq} f_{\fp}(x^{f(\fp|\fq)},\rho_0;E/L) = f_{\fq}(x,\Ind_{\Gal(E/L)}^{\Gal(E/K)}(\rho_0); E/K).$$ 
    \end{itemize}
\end{lem}
\begin{proof}
   This result is due to Artin \cite{Artin_30} however, we refer the reader to Cogdell's paper \cite{Cogdell_07} for a modern presentation.
\end{proof}

\begin{prop}\label{Artin sym function properties}
Suppose $E/L/K$ is a tower of finite Galois extensions,  $\rho,\gamma$ are two representations of $\Gal(L/K)$, and $\rho_0$ is a representation of $\Gal(E/L).$ The Artin symmetric functions satisfy the following properties:
    \begin{itemize}
        \item $$\sL(\fX_K,\rho;L/K)\sL(\fX_K,\gamma;L/K) = \sL(\fX_K,\rho \oplus \gamma;L/K)$$
        \item $$\sL(\fX_K,\rho;L/K) = \sL(\fX_K,\iota_{\Gal(L/K)}^{\Gal(E/K)}(\rho);E/K)$$ 
        \item $$\mathcal{N}_{L/K}\left(\sL(\fX_L,\rho_0;E/L)\right) = \sL(\fX_{K},\Ind_{\Gal(E/L)}^{\Gal(E/K)}(\rho_0); E/K).$$ 
    \end{itemize}
\end{prop}
\begin{proof}
    We may use Lemma \ref{Artin's lemma} directly. First, 
    \begin{align*}
        \sL(\fX_K,\rho;L/K)\sL(\fX_K,\gamma;L/K) &= \prod_{\fq \in \mathbb{P}_K}\prod_{i \geq 1} \det(\Id_{V^{I(\mathfrak{q})}}- x_i^{(\fq)}\rho(\sigma_{\mathfrak{q}})|_{V^{I(\mathfrak{q})}})^{-1} \prod_{\fq \in \mathbb{P}_K}\prod_{i \geq 1} \det(\Id_{V^{I(\mathfrak{q})}}- x_i^{(\fq)}\gamma(\sigma_{\mathfrak{q}})|_{V^{I(\mathfrak{q})}})^{-1}\\
        &= \prod_{\fq \in \mathbb{P}_K}\prod_{i \geq 1} \det(\Id_{V^{I(\mathfrak{q})}}- x_i^{(\fq)}\rho(\sigma_{\mathfrak{q}})|_{V^{I(\mathfrak{q})}})^{-1} \det(\Id_{V^{I(\mathfrak{q})}}- x_i^{(\fq)}\gamma(\sigma_{\mathfrak{q}})|_{V^{I(\mathfrak{q})}})^{-1}\\
        &= \prod_{\fq \in \mathbb{P}_K}\prod_{i \geq 1} f_{\fq}(x_i^{(\fq)},\rho;L/K)^{-1}f_{\fq}(x_i^{(\fq)},\gamma;L/K)^{-1}\\
        &=  \prod_{\fq \in \mathbb{P}_K}\prod_{i \geq 1} f_{\fq}(x_i^{(\fq)},\rho\oplus \gamma;L/K)^{-1}\\
        &= \sL(\fX_K,\rho \oplus \gamma;L/K).\\
    \end{align*}
    Next, 
    \begin{align*}
    \sL(\fX_K,\rho;L/K) &= \prod_{\fq \in \mathbb{P}_K}\prod_{i \geq 1} \det(\Id_{V^{I(\mathfrak{q})}}- x_i^{(\fq)}\rho(\sigma_{\mathfrak{q}})|_{V^{I(\mathfrak{q})}})^{-1} \\
    &= \prod_{\fq \in \mathbb{P}_K}\prod_{i \geq 1} f_{\fq}(x_i^{(\fq)},\rho;L/K)^{-1}\\
    &= \prod_{\fq \in \mathbb{P}_K}\prod_{i \geq 1}f_{\fq}(x_i^{(\fq)},\iota_{\Gal(L/K)}^{\Gal(E/K)}(\rho);E/K)^{-1}\\
    &= \sL(\fX_K,\iota_{\Gal(L/K)}^{\Gal(E/K)}(\rho);E/K).\\
    \end{align*}
Lastly, 
    \begin{align*}
        \sL(\fX_{K},\Ind_{\Gal(E/L)}^{\Gal(E/K)}(\rho_0); E/K) &= \prod_{\fq \in \mathbb{P}_K}\prod_{i \geq 1} f_{\fq}(x_i^{(\fq)},\Ind_{\Gal(E/L)}^{\Gal(E/K)}(\rho_0);E/K)^{-1} \\
        &= \prod_{\fq \in \mathbb{P}_K}\prod_{i \geq 1} \prod_{\fp|\fq} f_{\fp}((x_i^{(\fq)})^{f(\fp|\fq)},\rho_0;E/L)^{-1} \\
        &= \mathcal{N}_{L/K}\left( \prod_{\fq \in \mathbb{P}_K}\prod_{i \geq 1} \prod_{\fp|\fq} f_{\fp}(x_i^{(\fp)},\rho_0;E/L)^{-1} \right)\\
        &= \mathcal{N}_{L/K}\left( \prod_{\fp \in \mathbb{P}_L}\prod_{i \geq 1} f_{\fp}(x_i^{(\fp)},\rho_0;E/L)^{-1} \right)\\
        &= \mathcal{N}_{L/K}\left( \sL(\fX_L,\rho_0;E/L) \right).\\
    \end{align*}

\end{proof}

In this algebraic setting we can realize the relative Dedekind symmetric functions as the Artin symmetric functions for the regular representation of $\Gal(L/K).$

\begin{prop}
    For any finite Galois extension $L/K$ of number fields 
    $$\sL(\fX_K,\Ind_{\Gal(L/L)}^{\Gal(L/K)}(\mathbbm{1});L/K) = \zeta_{L/K}[\fX_K].$$
\end{prop}
\begin{proof}
    \begin{align*}
        &\sL(\fX_K,\Ind_{\Gal(L/L)}^{\Gal(L/K)}(\mathbbm{1});L/K)\\
        &= \mathcal{N}_{L/K}\left(\sL(\fX_L,\mathbbm{1};L/L)\right) \\
        &= \mathcal{N}_{L/K}\left(\zeta_{L}[\fX_L]\right)\\
        &= \zeta_{L/K}[\fX_K].\\
    \end{align*}
\end{proof}

Using the properties in Proposition \ref{Artin sym function properties} we derive a factorization of the relative Dedekind symmetric function $\zeta_{L/K}$ in terms of the Artin symmetric functions of the irreducible representations of $\Gal(L/K).$

\begin{cor}\label{Artin sym functions factorization}
    For any finite Galois extension $L/K$ of number fields
    $$\zeta_{L/K}[\fX_K] = \zeta_{K}[\fX_K] \prod_{\rho \neq \mathbbm{1}} \sL(\fX_K,\rho;L/K)^{\dim(\rho)}$$
    where $\rho \neq \mathbbm{1}$ ranges over all equivalence classes of nontrivial irreducible complex representations of $\Gal(L/K).$
\end{cor}
\begin{proof}
    \begin{align*}
        \zeta_{L/K}[\fX_K] &= \sL(\fX_K,\Ind_{\Gal(L/L)}^{\Gal(L/K)}(\mathbbm{1});L/K)\\
        &= \sL(\fX_K,\Ind_{\Gal(L/L)}^{\Gal(L/K)}(\mathbbm{1});L/K) \\
        &= \sL(\fX_K,\bigoplus_{\rho} \rho^{\dim(\rho)};L/K)\\
        &= \prod_{\rho} \sL(\fX_K,\rho ;L/K)^{\dim(\rho)} \\
        &= \sL(\fX_K,\mathbbm{1};L/K) \prod_{\rho \neq \mathbbm{1}} \sL(\fX_K,\rho;L/K)^{\dim(\rho)}\\
        &= \zeta_{K}[\fX_K] \prod_{\rho \neq \mathbbm{1}} \sL(\fX_K,\rho;L/K)^{\dim(\rho)}.\\
    \end{align*}
\end{proof}

Using a classical theorem of Brauer we are able to express the Artin symmetric functions $\sL(\fX_K,\rho;L/K)$ for a general Galois representation $\rho$ in terms of $1$-dimensional characters of subgroups $\Gal(L/K_i)$ for intermediate extensions $K \subseteq K_i \subseteq L.$

\begin{thm}[Brauer]\label{Brauer's thm}
    Let $G$ be a finite group and $\rho$ a finite dimensional complex representation of $G.$ Then there exist (not necessarily distinct) subgroups $H_1,\ldots,H_r \leq G$ , integers $a_1,\ldots , a_r \in \mathbb{Z}$, and characters $\chi_i:H_i\rightarrow \mathbb{C}^{*}$ such that in the representation ring of $G,$
    $$\rho = \sum_{i=1}^{r} a_i \Ind_{H_i}^{G} \chi_i.$$ 
\end{thm}

\begin{prop}
    Let $\rho: \Gal(L/K) \rightarrow \GL(V)$ be a finite dimensional complex Galois representation for number fields $L/K.$ Then there exist (not necessarily distinct) intermediate extensions $K \subseteq  K_1,\ldots,K_r \subseteq L ,$  integers $a_1,\ldots , a_r \in \mathbb{Z}$, and $1$-dimensional characters $\chi_i:\Gal(L/K_{i}) \rightarrow \mathbb{C}^{*}$ such that
    $$\sL(\fX_K,\rho;L/K) = \prod_{i =1}^{r} \mathcal{N}_{K_i/K}\left( \sL(\fX_{K_i},\chi_i;L/K_i)\right)^{a_i}.$$
\end{prop}
\begin{proof}
    Use Brauer's theorem \ref{Brauer's thm} and the induction property from Proposition \ref{Artin sym function properties}.
\end{proof}

\subsection{Monomial and Schur Expansions}

For the remainder of the next two sections assume the notation from Definition \ref{main def for Artin sym}. Importantly, we have defined the groups $D(\fq),I(\fq)$ by arbitrarily \textbf{\textit{choosing}} primes $\fp | \fq$ for each $\fq.$ These choices will affect the presentation of the objects in the remainder of this paper but will not fundamentally change the results. In this section we will compute the expansion coefficients of the Artin symmetric functions $\sL(\fX_K,\rho;L/K)$ with respect to the monomial $m_{\lambda^{\bullet}}$ and Schur $s_{\lambda^{\bullet}}$ bases of $A_{K}.$

\begin{defn}
    Define the Hall inner product on $A_{K}$ as
$$\langle p_{\lambda^{\bullet}}, p_{\mu^{\bullet}} \rangle_{H} := \delta_{\lambda^{\bullet},\mu^{\bullet}} \prod_{\fq \in \mathbb{P}_{K}}\prod_{i \geq 1}(m_{i}(\lambda^{(\fq)}) i^{m_{i}(\lambda^{(\fq)})}).$$
\end{defn}

The next lemma will be important later.

\begin{lem}\label{Cauchy Identity Lemma}
    Let $n \geq 1.$ For each $\fq \in \mathbb{P}_K$ consider a set $\{y_1^{(\fq)},\ldots, y_{n}^{(\fq)}\}$ of commuting free variables which mutually commute between the variable sets $y_i^{(\fq)},y_{j}^{(\fp)}$ and commute with $\mathfrak{X}_K.$  Consider any pair $\{ u_{\lambda}, v_{\lambda} \}_{\lambda}$ of dual bases of $\Lambda$ with respect to $\langle -,-\rangle_{H}.$ Define $v_{\lambda^{\bullet}}[\mathfrak{X}_{K}]:= \prod_{\fq \in \mathbb{P}_K}v_{\lambda^{(\fq)}}[\mathfrak{X}_{\fq}]$. Then 
    $$\prod_{\fq \in \mathbb{P}_K}\prod_{i=1}^{n}\prod_{j} (1-y_{i}^{(\fq)}x_{j}^{(\fq)})^{-1} = \sum_{\lambda^{\bullet}} \prod_{\fq \in \mathbb{P}_K} \left(u_{\lambda^{(\fq)}}(y_1^{(\fq)},\ldots,y_{n}^{(\fq)}) \right) v_{\lambda^{\bullet}}[\mathfrak{X}_{\fq}].$$
\end{lem}

\begin{defn}
 Define the multi-set of elements of $\mathbb{Q}^{\text{ab}}$, $A^{(\rho,\fq)}:= \{\alpha_{1}^{(\rho,\fq)},\ldots, \alpha_{d^{(\rho,\fq)}}^{(\rho,\fq)}\} $, to be the eigenvalues of $\rho(\sigma_{\fq})|_{V^{I(\fq)}}$ (with multiplicity) where $d^{(\rho,\fq)}:= \dim V^{I(\fq)}.$ Given a symmetric function $f \in \Lambda$ we write $f|_{A^{(\rho,\fq)}}$ for the evaluation $f(\alpha_{1}^{(\rho,\fq)},\ldots, \alpha_{d^{(\rho,\fq)}}^{(\rho,\fq)})$ of $f$ at the multi-set $A^{(\rho,\fq)}.$
\end{defn}

We now see that we have the following:

\begin{prop}\label{Cauchy identity for Artin}
    Let $\{u_{\lambda},v_{\lambda}\}_{\lambda}$ be any pair of dual bases for $\Lambda$ with respect to $\langle-,-\rangle_{H}.$ Then 
    $$\sL(\fX_K,\rho;L/K) = \sum_{\lambda^{\bullet}} \prod_{\fq \in \mathbb{P}_K}\left( u_{\lambda^{(\fq)}}|_{A^{(\rho,\fq)}} \right) v_{\lambda^{\bullet}}[\mathfrak{X}_K].$$
\end{prop}
\begin{proof}
    Each characteristic polynomial factor $\det(\Id_{V^{I(\mathfrak{q})}}- x_i^{(\mathfrak{q})}\rho(\sigma_{\mathfrak{q}})|_{V^{I(\mathfrak{q})}})$ of $\sL(\fX_K,\rho;L/K)$ factors as 
    $$\det(\Id_{V^{I(\mathfrak{q})}}- x_i^{(\mathfrak{q})}\rho(\sigma_{\mathfrak{q}})|_{V^{I(\mathfrak{q})}}) = \prod_{j=1}^{d^{(\rho,\fq)}}(1-x_i^{(\fq)}\alpha_{j}^{(\rho,\fq)}).$$ Therefore, 
    $$\sL(\fX_K,\rho;L/K) = \prod_{\fq \in \mathbb{P}_{K}} \prod_{i \geq 1}\det(\Id_{V^{I(\mathfrak{q})}}- x_i^{(\mathfrak{q})}\rho(\sigma_{\mathfrak{q}})|_{V^{I(\mathfrak{q})}})^{-1} = \prod_{\fq \in \mathbb{P}_{K}} \prod_{i \geq 1} \prod_{j = 1}^{d^{(\rho,\fq)}}(1-x_i^{(\fq)}\alpha_{j}^{(\rho,\fq)})^{-1}.$$ Since $d^{(\rho,\fq)} \leq \dim V$ we may set $n:= \dim V$ and define $y_j^{(\fq)}:= \alpha_{j}^{(\rho,\fq)}$ for $j \geq d^{(\rho,\fq)}$ and $y_j^{(\fq)} = 0$ for $d^{(\rho,\fq)}< j \leq n. $ Lemma \ref{Cauchy Identity Lemma} now reads as 
    $$\sL(\fX_K,\rho;L/K) = \prod_{\fq \in \mathbb{P}_{K}} \prod_{i \geq 1} \prod_{j = 1}^{d^{(\rho,\fq)}}(1-x_i^{(\fq)}\alpha_{j}^{(\rho,\fq)})^{-1} = \sum_{\lambda^{\bullet}} \prod_{\fq \in \mathbb{P}_K}\left( u_{\lambda^{(\fq)}}|_{A^{(\rho,\fq)}} \right) v_{\lambda^{\bullet}}[\mathfrak{X}_K].$$
\end{proof}

We will use the above proposition repeatedly in order to obtain expansions for the Artin symmetric functions into different bases of $A_K.$ First, we will need some representation-theoretic definitions.

\begin{defn}
    If $H \leq G$ is a subgroup of a group $G$ we write $N_{G}(H)$ for the \textbf{\textit{normalizer}} of $H$ in $G.$ Define the profinite group $G_{L/K}$ as 
    $$G_{L/K}:= \prod_{\fq \in \mathbb{P}_K} N_{\Gal(L/K)}(I(\fq))/I(\fq).$$ For any group $G$, a representation $W$ of $G$, and a partition $\lambda,$ we write $\mathbb{S}_{\lambda}(W)$ for the representation of $G$ obtained by applying the \textbf{\textit{Schur functor}} $\mathbb{S}_{\lambda}$ to $W.$ Explicitly, 
    $$\mathbb{S}_{\lambda}(W):= W^{\otimes |\lambda|} \otimes_{\mathfrak{S}_{|\lambda|}} S^{\lambda}$$ where $S^{\lambda}$ is the Specht module associated to $\lambda$. We will write for $\lambda = (\lambda_1,\ldots,\lambda_r)$
    $$\eta_{\lambda}(W):= \mathbb{S}_{(\lambda_1)}(W)\otimes \ldots \otimes \mathbb{S}_{(\lambda_r)}(W).$$ Note that $\mathbb{S}_{(m)}(W)$ is simply the $m$-th symmetric power of $W.$
    
    For any factored representation $W = \bigotimes_{\fq \in \mathbb{P}_K} W_{\fq}$ of $G_{L/K}$ and $\lambda^{\bullet} \in \mathbb{Y}_{K}$ define the representation $\mathbb{S}_{\lambda^{\bullet}}(W)$ as 
    $$\mathbb{S}_{\lambda^{\bullet}}(W):= \bigotimes_{\fq \in \mathbb{P}_K} \mathbb{S}_{\lambda^{(\fq)}}(W_{\fq}).$$ Similarly, we define 
    $$\eta_{\lambda^{\bullet}}(W):= \bigotimes_{\fq \in \mathbb{P}_K} \eta_{\lambda^{(\fq)}}(W_{\fq}).$$ Note that we are following the usual conventions for infinite tensor products:
    $$\bigotimes_{\fq \in \bP_K}:= \lim_{\substack{S\subset \bP_K\\ |S| <\infty}} \bigotimes_{\fq \in S}.$$
    
\end{defn}

\begin{remark}
    The group $G_{L/K}$ can be decomposed as 
    $$\prod_{\fq \nmid \mathscr{D}_{L/K}} \Gal(L/K) \times \prod_{\fq \mid \mathscr{D}_{L/K}} N_{\Gal(L/K)}(I(\fq))/I(\fq).$$ Therefore, the presentation of the group $G_{L/K}$ as $\prod_{\fq \in \mathbb{P}_K} N_{\Gal(L/K)}(I(\fq))/I(\fq)$ depends on our choice of primes $\fp \in \mathbb{P}_{L}$ laying above $\fq \mid \mathscr{D}_{L/K}$ but up to isomorphism $G_{L/K}$ is independent of this choice. Note that $I(\fq) $ is a normal subgroup of $D(\fq)$ so that $D(\fq) \leq N_{\Gal(L/K)}(I(\fq)).$
\end{remark}

\begin{defn}
     Define the element $$\sigma_{K}:= (\sigma_{\fq})_{\fq\in \mathbb{P}_K} \in G_{L/K}.$$ Define the representation $V_{\fq}$ of $G_{L/K}$ to be the action of $G_{L/K}$ which is trivial in every component except for the action of $N_{\Gal(L/K)}(I(\fq))/I(\fq)$ on $V^{I(\fq)}$ in the $\fq$-component. Define the infinite dimensional representation of $G_{L/K}$
    $$W_{\rho}:= \bigotimes_{\fq \in \mathbb{P}_K} V_{\fq}$$ in which we interpret the infinite tensor in the usual way.
\end{defn}

\begin{remark}
    Note that since
    $$\mathbb{S}_{\lambda^{\bullet}}(W_{\rho}) = \bigotimes_{\fq \in \mathbb{P}_K} \mathbb{S}_{\lambda^{(\fq)}}(V_{\fq})$$ and $\lambda^{(\fq)} = \emptyset$ for almost all $\fq,$ the space $\mathbb{S}_{\lambda^{\bullet}}(W_{\rho})$ is actually a finite dimensional representation of $G_{L/K}$ for all $\lambda^{\bullet}.$ The same holds for $\eta_{\lambda^{\bullet}}(W_{\rho}).$ Therefore, computing the trace of $\sigma_{K}$ acting on any of these representations is well-defined.
\end{remark}

We are now able to find the monomial expansion of the Artin symmetric functions.

\begin{prop}\label{monomial expansion}
    $$\sL(\fX_K,\rho;L/K) = \sum_{\lambda^{\bullet}} \Tr_{\eta_{\lambda^{\bullet}}(W_{\rho})}(\sigma_{K}) m_{\lambda^{\bullet}}[\mathfrak{X}_K]$$
\end{prop}
\begin{proof}
By applying Proposition \ref{Cauchy identity for Artin} to the monomial symmetric functions we find:

$$\sL(\fX_K,\rho;L/K) = \sum_{\lambda^{\bullet}} \prod_{\fq \in \mathbb{P}_K}\left( h_{\lambda^{(\fq)}}|_{A^{(\rho,\fq)}} \right) m_{\lambda^{\bullet}}[\mathfrak{X}_K].$$
We will now show that the values $\prod_{\fq \in \mathbb{P}_K}\left( h_{\lambda^{(\fq)}}|_{A^{(\rho,\fq)}} \right)$ are actually explicit character values associated to the representation $\eta_{\lambda^{\bullet}}.$

Firstly,
    \begin{align*}
        \Tr_{\eta_{\lambda^{\bullet}}(W_{\rho})}(\sigma_{K}) &= \Tr_{\bigotimes_{\fq \in \mathbb{P}_K} \eta_{\lambda^{(\fq)}}(V_{\fq})}(\sigma_{K})\\
        &= \prod_{\fq \in \mathbb{P}_K} \Tr_{\eta_{\lambda^{(\fq)}}(V_{\fq})}(\sigma_{\fq}).\\
    \end{align*}
Now we have 
$$\eta_{\lambda^{(\fq)}}(V_{\fq}) = \bigotimes_{i} \mathbb{S}_{(\lambda_{i}^{(\fq)})}(V)$$ so 
$$\Tr_{\eta_{\lambda^{(\fq)}}(V_{\fq})}(\sigma_{\fq}) = \prod_{i} h_{\lambda_{i}^{(\fq)}}|_{A^{(\rho,\fq)}} = h_{\lambda^{(\fq)}}|_{A^{(\rho,\fq)}}.$$
Therefore, 
$$\Tr_{\eta_{\lambda^{\bullet}}(V)}(\sigma_{K}) = \prod_{\fq \in \mathbb{P}_K}\left( h_{\lambda^{(\fq)}}|_{A^{(\rho,\fq)}} \right).$$

\end{proof}

We may also compute the Schur expansion readily:

\begin{lem}\label{schur functor lemma}
    If $g\in G$ acts on a representation $W$ of $G$ with eigenvalues $\{\alpha_1,\ldots, \alpha_n\}$ (with multiplicity) then 
    $$\Tr_{\mathbb{S}_{\lambda}(W)}(g) = s_{\lambda}(\alpha_1,\ldots, \alpha_n).$$
\end{lem}
\begin{proof}
   Schur-Weyl duality.
\end{proof}

\begin{prop}\label{schur expansion prop}
    $$\sL(\fX_K,\rho;L/K) = \sum_{\lambda^{\bullet}} \Tr_{\mathbb{S}_{\lambda^{\bullet}}(W_{\rho})}(\sigma_{K}) s_{\lambda^{\bullet}}[\mathfrak{X}_{K}].$$
\end{prop}
\begin{proof}
    First we apply Lemma \ref{Cauchy Identity Lemma} to the Schur functions, which are self dual with respect to $\langle -, -\rangle_{H}$, to find 
    $$\sL(\fX_K,\rho;L/K) = \sum_{\lambda^{\bullet}} \prod_{\fq \in \mathbb{P}_K}\left( s_{\lambda^{(\fq)}}|_{A^{(\rho,\fq)}} \right) s_{\lambda^{\bullet}}[\mathfrak{X}_K].$$
    Now we compute using Lemma \ref{schur functor lemma}:
    \begin{align*}
        \Tr_{\mathbb{S}_{\lambda^{\bullet}}(W_{\rho})}(\sigma_{K}) &= \Tr_{\bigotimes_{\fq \in \mathbb{P}_K} \mathbb{S}_{\lambda^{(\fq)}}(V_{\fq})}(\sigma_{K})\\
        &= \prod_{\fq \in \mathbb{P}_K} \Tr_{\mathbb{S}_{\lambda^{(\fq)}}(V_{\fq})}(\sigma_{K})\\
        &= \prod_{\fq \in \mathbb{P}_K} \Tr_{\mathbb{S}_{\lambda^{(\fq)}}(V)}(\sigma_{\fq})\\
        &= \prod_{\fq \in \mathbb{P}_K} s_{\lambda^{(\fq)}}|_{A^{(\rho,\fq)}}.\\
    \end{align*}
\end{proof}

\begin{remark}
    A similar computation may be performed to expand $\sL(\fX_K,\rho;L/K)$ into the forgotten symmetric function basis $f_{\lambda^{\bullet}}$ which yields an analogous result to Proposition \ref{monomial expansion} except that the Schur functors $\mathbb{S}_{(m)}$ are replaced with $\mathbb{S}_{(1^m)}$ in the definition of $\eta_{\lambda^{\bullet}}(W_{\rho}).$ The functors $\mathbb{S}_{(1^m)}$ are the exterior power functors $V \rightarrow \Lambda^{m}(V).$ 
\end{remark}

\subsection{Hall-Littlewood Expansion}

In this section we compute the Hall-Littlewood expansions of the Artin symmetric functions. This will require significantly more work and rely fundamentally on the Schur positivity of the modified Hall-Littlewood functions. The coefficients we will obtain in the expansion will be the character values of certain finite dimensional representations of the compact group $G_{L/K}$ built using a form of Schur-Weyl duality for the modified Hall-Littlewood polynomials. Importantly, the Lascoux-Sch\"{u}tzenberger cocharge formula for the modified Kostka-Foulkes polynomials will give us a direct way to express these trace values in terms of the more familiar Schur functor trace values seen in Proposition \ref{schur expansion prop}.

\begin{defn}\label{modified HL def}
    For a partition $\lambda$ and a variable $t$ define the \textbf{\textit{modified Hall-Littlewood function}} as 
    $$\widetilde{H}_{\lambda}[\fX;t]:= t^{n(\lambda)} Q_{\lambda}\left[\frac{\fX}{1-t^{-1}};t^{-1} \right].$$
\end{defn}

Here the plethystic evaluation $F\left[\frac{\fX}{1-t^{-1}} \right]$ is given explicitly by 
$$p_r\left[\frac{\fX}{1-t^{-1}} \right]:= \frac{p_r[\fX]}{1-t^{-r}}$$ and extended to $\mathbb{Q}(t) \otimes \Lambda$ accordingly.

\begin{remark}
    The modified Hall-Littlewood functions are \textit{Schur positive}:
    $$\widetilde{H}_{\lambda}[\fX;t] = \sum_{\mu} \widetilde{K}_{\mu, \lambda}(t) s_{\mu}[\fX]$$ where $ \widetilde{K}_{\mu, \lambda}(t) \in \mathbb{Z}_{\geq 0}[t] $ are the \textbf{\textit{modified Kostka-Foulkes polynomials}}. These polynomials are known \cite{LS_78} to have the following combinatorial interpretation:
\begin{equation}\label{cocharge formula}
    \widetilde{K}_{\mu, \lambda}(t) = \sum_{T \in \text{SSYT}(\mu, \lambda)} t^{\text{cocharge}(T)} 
\end{equation}
 where $ \text{SSYT}(\mu, \lambda)$ denotes the set of semi-standard Young tableaux of shape $\lambda$ and content $\mu$ and $\text{cocharge}(T)$ is a certain non-negative combinatorial statistic associated to $T.$
\end{remark}

\begin{defn}
    Define the \textbf{\textit{modified Hall-Littlewood functors}} as functors $$\widetilde{\mathbb{H}}_{\lambda}(-;m): \text{Vect.} \rightarrow \text{Vect.}$$ with $\lambda$ a partition and $m \in \mathbb{N}$ a parameter as 
    $$\widetilde{\mathbb{H}}_{\lambda}(W;m):=  \bigoplus_{\textrm{content}(T) = \lambda} \mathbb{S}_{\sh(T)}(W)^{\oplus m^{\textrm{cocharge}(T)}}$$ 
    where $T$ ranges over all semi-standard Young tableaux of \textit{any} \textbf{\textit{shape}} $\sh(T)$ and \textbf{\textit{content partition}} $\textrm{content}(T) = \lambda.$ For any group $G$ we may consider $\widetilde{\mathbb{H}}_{\lambda}(-;m)$ as a functor 
    $$\widetilde{\mathbb{H}}_{\lambda}(-;m): G-\text{Mod} \rightarrow G-\text{Mod}.$$ 
\end{defn}

The following simple lemma is an application of Schur-Weyl duality applied to the polynomial functor corresponding to the symmetric polynomial $\widetilde{H}_{\lambda}(x_1,\ldots, x_d;m).$

\begin{lem}\label{modified HL lemma}
    If $g \in G$ and $W$ is a finite-dimensional representation of $G$ such that $g$ acting on $W$ has eigenvalues $\{\alpha_1,\ldots , \alpha_d\}$ (with multiplicity) then for any partition $\lambda$ and $m \in \mathbb{N}$ 
    $$\Tr_{\widetilde{\mathbb{H}}_{\lambda}(W;m)} (g)  = \widetilde{H}_{\lambda}(\alpha_1,\ldots, \alpha_d;m).$$
\end{lem}
\begin{proof}
    Immediate from the cocharge formula \eqref{cocharge formula} and Lemma \ref{schur functor lemma}. 
\end{proof}

\begin{defn}
    Given any factored representation $U =  \bigotimes_{\fq \in \mathbb{P}_{K}} U_{\fq}$ for $G_{L/K}$ and any $\lambda^{(\bullet)} \in \mathbb{Y}_K$ we define the representation of $G_{L/K}$
    $$\widetilde{\mathbb{H}}_{\lambda^{\bullet}}(U):= \bigotimes_{\fq \in \mathbb{P}_{K}} \widetilde{\mathbb{H}}_{\lambda^{(\fq)}}(U_{\fq};N(\fq)).$$
\end{defn}

We require the following lemma:

\begin{lem}
For any parameter $t$ which is not a root of unity
    $$\Exp[\fX \fY] = \sum_{\lambda} t^{-n(\lambda)} P_{\lambda}[\fX;t^{-1}] \widetilde{H}_{\lambda}[\fY;t].$$ As such 
    $$\langle t^{-n(\lambda)} P_{\lambda}[\fX;t^{-1}], \widetilde{H}_{\mu}[\fX;t] \rangle_{H} = \delta_{\lambda,\mu}.$$
\end{lem}
\begin{proof}
    First, substitute $\fX \fY = \frac{(1-t^{-1})\fX\fY}{1-t^{-1}}$ into the plethystic exponential $\Exp[\fX]$ and apply the $t$-Hall Cauchy Identity \ref{t cauchy identity}:
    $$\Exp[\fX \fY] = \sum_{\lambda}  P_{\lambda}[\fX;t^{-1}] Q_{\lambda}\left[ \frac{\fY}{1-t^{-1}};t^{-1}\right]$$
    Lastly, use Definition \ref{modified HL def} to write 
    $$Q_{\lambda}\left[ \frac{\fY}{1-t^{-1}};t^{-1}\right] = t^{-n(\lambda)} \widetilde{H}_{\lambda}[\fY;t].$$
\end{proof}

We now give the arithmetic Hall-Littlewood expansions of the $\sL(\fX_K,\rho;L/K)$.

\begin{thm}\label{main thm HL expansion}
    $$\sL(\fX_K,\rho;L/K) = \sum_{\lambda^{\bullet} \in \mathbb{Y}_K}  \Tr_{\widetilde{\mathbb{H}}_{\lambda^{\bullet}}(W_{\rho})}(\sigma_{K})  g(\lambda^{\bullet})^{-1} P_{\lambda^{\bullet}}[\mathfrak{X}_K] $$
\end{thm}
\begin{proof}
    First, we use Proposition \ref{Cauchy identity for Artin} to see that 
    $$\sL(\fX_K,\rho;L/K) = \sum_{\lambda^{\bullet} \in \mathbb{Y}_{K}} \prod_{\fq \in \mathbb{P}_K} \left( \widetilde{H}_{\lambda^{(\fq)}}[\fX^{(\fq)};N(\fq)]|_{A^{(\rho,\fq)}} \right) g(\lambda^{\bullet})^{-1}P_{\lambda^{(\fq)}}[\fX^{(\fq)};N(\fq)^{-1}].$$ Next using Lemma \ref{modified HL lemma} we get 
    $$\sum_{\lambda^{\bullet}} \prod_{\fq \in \mathbb{P}_K} \left( \widetilde{H}_{\lambda^{(\fq)}}[\fX^{(\fq)};N(\fq)]|_{A^{(\rho,\fq)}} \right) g(\lambda^{\bullet})^{-1}P_{\lambda^{(\fq)}}[\fX^{(\fq)};N(\fq)^{-1}] = \sum_{\lambda^{\bullet} \in \mathbb{Y}_K} \prod_{\fq \in \mathbb{P}_K} \Tr_{\widetilde{\mathbb{H}}_{\lambda^{(\fq)}}(V_{\fq};N(\fq))}(\sigma_{\fq}) g(\lambda^{\bullet})^{-1} P_{\lambda^{\bullet}}[\fX_K].$$ Since 
    $$\widetilde{\mathbb{H}}_{\lambda^{\bullet}}(W_{\rho}) = \bigotimes_{\fq \in \mathbb{P}_K}\widetilde{\mathbb{H}}_{\lambda^{(\fq)}}(V_{\fq};N(\fq))$$ we know that 
    $$\prod_{\fq \in \mathbb{P}_K}\Tr_{\widetilde{\mathbb{H}}_{\lambda^{(\fq)}}(V_{\fq};N(\fq))}(\sigma_{\fq}) = \prod_{\fq \in \mathbb{P}_K}\Tr_{\widetilde{\mathbb{H}}_{\lambda^{(\fq)}}(V_{\fq};N(\fq))}(\sigma_{K}) = \Tr_{\widetilde{\mathbb{H}}_{\lambda^{\bullet}}(W_{\rho})}(\sigma_K).$$ Therefore, 
    $$\sL(\fX_K,\rho;L/K) = \sum_{\lambda^{\bullet} \in \mathbb{Y}_K}  \Tr_{\widetilde{\mathbb{H}}_{\lambda^{\bullet}}(W_{\rho})}(\sigma_{K})  g(\lambda^{\bullet})^{-1} P_{\lambda^{\bullet}}[\mathfrak{X}_K]$$
\end{proof}

To finish this section we give a description of the $G_{L/K}$ modules $\widetilde{\mathbb{H}}_{\lambda^{\bullet}}(W_{\rho})$ in terms of the Schur functors.

\begin{defn}
    A \textbf{\textit{semi-standard arithmetic Young tableau}} is a sequence $T^{\bullet} = (T^{(\fq)})_{\fq \in \mathbb{P}_{K}}$ of semi-standard Young tableaux $T^{(\fq)}$ such that for all but finitely many $\fq \in \mathbb{P}_{K}$, $\mathrm{sh}(T^{(\fq)}) = \emptyset.$ The shape of $T^{\bullet}$ is given by $\mathrm{sh}(T^{\bullet}):= (\mathrm{sh}(T^{(\fq)}))_{\fq \in \mathbb{P}_K} \in \mathbb{Y}_{K}$ and the content of $T^{\bullet}$ is $\mathrm{content}(T^{\bullet}) := (\mathrm{content}(T^{(\fq)}))_{\fq \in \mathbb{P}_{K}}.$ We define the \textbf{\textit{arithmetic cocharge}} as 
    $$\mathrm{cocharge}(T^{\bullet}):= \prod_{\fq \in \mathbb{P}_K} N(\fq)^{\mathrm{cocharge}(T^{(\fq)})} \in \mathbb{Y}_{K}.$$
\end{defn}

\begin{lem}\label{HL into Schur functor lemma}
Given any factored representation $U = \bigotimes_{\fq \in \mathbb{P}_K} U_{\fq}$ of $G_{L/K}$ 
 $$\widetilde{\mathbb{H}}_{\lambda^{\bullet}}(U) = \bigoplus_{\mathrm{content}(T^{\bullet}) = \lambda^{\bullet}} \mathbb{S}_{\mathrm{sh}(T^{\bullet})}(U)^{\oplus \mathrm{cocharge}(T^{\bullet})}.$$
\end{lem}
\begin{proof}
    \begin{align*}
    \widetilde{\mathbb{H}}_{\lambda^{\bullet}}(U) &= \bigotimes_{\fq \in \mathbb{P}_K} \widetilde{\mathbb{H}}_{\lambda^{(\fq)}}(U_{\fq};N(\fq))\\
    &= \bigotimes_{\fq \in \mathbb{P}_K} \bigoplus_{\mathrm{content}(T^{(\fq)}) = \lambda^{(\fq)}} \mathbb{S}_{\mathrm{sh}(T^{(\fq)})}(U_{\fq})^{\oplus N(\fq)^{\mathrm{cocharge}(T^{(\fq)})}}\\
    &= \bigoplus_{\mathrm{content(T^{\bullet})}=\lambda^{\bullet}} \mathbb{S}_{\mathrm{sh}(T^{\bullet})}(U)^{\oplus \prod_{\fq \in \mathbb{P}_K} N(\fq)^{\mathrm{cocharge}(T^{(\fq)})}}\\
    &= \bigoplus_{\mathrm{content(T^{\bullet})}=\lambda^{\bullet}} \mathbb{S}_{\mathrm{sh}(T^{\bullet})}(U)^{\oplus \mathrm{cocharge}(T^{\bullet})}.\\
    \end{align*}
\end{proof}

By expanding the trace $\Tr_{\widetilde{\mathbb{H}}_{\lambda^{\bullet}}(W_{\rho})}(\sigma_{K})$ according to Lemma \ref{HL into Schur functor lemma} in the case when $U = W_{\rho}$ we find the following:

\begin{prop}
   $$ \Tr_{\widetilde{\mathbb{H}}_{\lambda^{\bullet}}(W_{\rho})}(\sigma_{K}) = \sum_{\mathrm{content}(T^{\bullet}) = \lambda^{\bullet}} \mathrm{cocharge}(T^{\bullet}) \Tr_{\mathbb{S}_{\mathrm{sh}(T^{\bullet})}(W_{\rho})}(\sigma_{K})$$
\end{prop}

\section{Dirichlet Series from Artin Symmetric Functions}

In this section we will construct Dirichlet series from the Artin symmetric functions. We will first construct functions on $\GL_n(\widehat{\sO}_K) \backslash \GL_n(\mathbb{A}_{K}^{f}) / \GL_n(\widehat{\sO}_K)$ by using the Satake map $\ref{Satake iso}.$ Then we compute the Mellin transforms of these functions and study their analytic properties. It will turn out that the resulting Dirichlet series are products of shifts of Artin L-functions. We show that after taking a re-normalized limit we obtain an expansion for the infinite product of shifted Artin L-functions in terms of the trace values seen before.

\subsection{Spherical Hecke Algebra}

We refer the reader to Bump's lectures \cite{Bump} for a detailed introduction to Hecke algebras.

\begin{defn}
    We will denote by $H_n(K)$ and $G_n(K)$ the locally compact topological groups $H_n(K):= \GL_n(\widehat{\sO}_K)$ and $G_n(K):= \GL_n(\mathbb{A}_{K}^{f}).$ Note that $H_n(K)$ is compact. Define $G_n^{+}(K)$ to be the closed sub-monoid of $G_n(K)$ consisting of all $g \in G_n(K)$ whose entries are in $\widehat{\sO}_K.$ Let $\mu_{K,n}$ denote the unique Haar measure on $G_n(K)$ such that $\mu_{K,n}(H_n(K)) = 1.$ This measure is both left and right invariant. We will denote by $\mathcal{H}_n(K)$ the \textit{\textbf{spherical Hecke algebra}} over $\mathbb{Q}^{\text{ab}}$ of functions $f: G_n^{+}(K) \rightarrow \mathbb{Q}^{\text{ab}}$ which satisfy the following properties:
    \begin{itemize}
        \item $f$ is continuous where $\mathbb{Q}^{\text{ab}}$ is given the discrete topology
        \item $f$ has compact support
        \item $f$ is $H_n(K)$-biinvariant i.e. 
        $f(axb) = f(x)$ for all $a,b \in H_n(K)$ and $x \in G_n^{+}(K).$
    \end{itemize}
    The multiplication on $\mathcal{H}_n(K)$ is given by convolution:
    \begin{equation}\label{convolution eq}
        f*g(x):= \int_{G_n^{+}(K)} f(xy^{-1})g(y) d\mu_{K,n}.
    \end{equation}
     Define the pairing $\langle-,-\rangle_n:\mathcal{H}_n(K) \times \mathcal{H}_n(K) \rightarrow \mathbb{Q}^{\text{ab}} $ by 
    $$\langle f, g \rangle_n:= \int_{G_n^{+}(K)} f(x)g(x)^{*} d\mu_{K,n}.$$
\end{defn}

\begin{remark}
    The groups $(G_{n}(K),H_n(K))$ form a \textbf{\textit{Gelfand pair}} and as such the spherical Hecke algebra $\mathcal{H}_n(K)$ is \textit{commutative}. The standard proof of this fact in the $\GL$ case involves using the matrix transpose map $A \rightarrow A^{t}$ as an anti-involution of $G_{n}(K)$ on each double coset $H_{n}(K)gH_{n}(K)$ which has the effect of essentially swapping the role of $x$ and $y$ in the integral \eqref{convolution eq}.
\end{remark}

 \begin{defn}
    We will fix a set of uniformizers $\pi_{\fq} \in \sO_{K_{\fq}}$ for all $\fq \in \bP_K.$ The explicit choices of uniformizers will not matter ever in this paper so will we use this notation without ambiguity. 
     Given $\lambda^{\bullet} \in \Par_K$ define $\ell(\lambda^{\bullet}):= \max_{\fq \in \bP_K} \ell(\lambda^{(\fq)})$ and let $\mathbb{Y}_{K}^{(n)}$ denote the set of all $\lambda^{\bullet} \in \mathbb{Y}_{K}$ with $\ell(\lambda^{\bullet}) \leq n$. For $\lambda^{\bullet} \in \Par_K^{(n)}$ define 
    $\pi_n^{\lambda^{\bullet}}:= \prod_{\fq \in \bP_K} \pi_{\fq}^{\lambda^{(\fq)}} \in G_{+}^{(n)}(K)$
where $$\pi_{\fq}^{\lambda^{(\fq)}}:= \begin{pmatrix}
\pi_{\fq}^{\lambda^{(\fq)}_1} & 0 & \ldots & 0\\
0 & \pi_{\fq}^{\lambda^{(\fq)}_2} & \ldots &  0\\
\vdots & \vdots &\ddots &  \vdots\\
 0     &    0    & \ldots  &  \pi_{\fq}^{\lambda^{(\fq)}_n} \\
\end{pmatrix}$$ and we extend each $\lambda^{(\fq)}$ with additional $0$'s on the right when needed.

We will write $$\mathbbm{1}_{\pi_n^{\lambda^{\bullet}}} \in \mathcal{H}_{n}(K)$$ for the indicator function on the double coset 
$$H_n(K)\pi_n^{\lambda^{\bullet}} H_{n}(K) \subset G_{n}^{+}(K).$$
 \end{defn}

\begin{remark}
The group $\GL_{n}(\mathbb{A}_{K}^{f})$ is isomorphic as a topological group to the \textbf{\textit{restricted direct product}} 
$$\GL_{n}(\mathbb{A}_{K}^{f}) \cong \prod_{\fq \in \mathbb{P}_K} (\GL_{n}(K_{\fq});\GL_{n}(\mathcal{O}_{K_{\fq}})) $$ of the locally compact groups $\GL_{n}(K_{\fq})$ with respect to the compact groups $\GL_{n}(\mathcal{O}_{K_{\fq}}).$ This means that every element $x \in \GL_{n}(\mathbb{A}_{K}^{f})$ may be thought of as a sequence $(x_{\fq})_{\fq \in \mathbb{P}_K}$ where each $x_{\fq} \in \GL_{n}(K_{\fq})$ and for all but finitely many $\fq,$ $x_{\fq} \in \GL_{n}(\mathcal{O}_{K_{\fq}}).$ Therefore, the double cosets of $H_{n}(K) \backslash G_{n}^{+}(K) / H_{n}(K)$ are exactly given by the set of $\pi_n^{\lambda^{\bullet}}$ for $\lambda^{\bullet} \in \mathbb{Y}_{K}^{(n)}$ and so the set 
$\{\mathbbm{1}_{\pi_n^{\lambda^{\bullet}}} ~|~ \lambda^{\bullet} \in \Par_{K}^{(n)} \}$ is a $\mathbb{Q}^{\text{ab}}$ basis for $\mathcal{H}_{n}(K).$
\end{remark}

\begin{defn}
    For $n \geq 1$ define $\mathcal{A}_{K,n}$ to be the specialization of the ring $\mathcal{A}_K$ to the variable sets $\fX^{(\fq, n)}:= x_1^{(q)}+\ldots + x_n^{(q)}.$ We will write $\fX_K^{(n)}$ for the collection of variable sets $(\fX^{(\fq,n)})_{\fq \in \mathbb{P}_K}$. We will write 
    $$\Xi_{K,n}:\mathcal{A}_K \rightarrow \mathcal{A}_{K,n}$$ for the natural quotient map. Similarly, we analogously define $A_{K,n}:= \Xi_{K,n}(A_{K})$. If $L/K$ is a finite extension then we will write 
    $$\mathcal{N}_{L/K}^{(n)}: \mathcal{A}_{L,n}\rightarrow \mathcal{A}_{K,n}$$ for the norm map defined by the substitution
    $$\fX^{(\fp,n)} \rightarrow p_{f(\fp|\fq)}[\fX^{(\fq,n)}]$$ whenever 
    $\fp|\fq.$
\end{defn}

The following result is a special case of the \textbf{\textit{Satake isomorphism}} applied to the $\GL_{n}(\mathbb{A}_{K}^{f})$. To be more specific, the Satake isomorphism applies to each local factor $\GL_{n}(K_{\fq})$ of $\GL_{n}(\mathbb{A}_{K}^{f})$ and we obtain the following result by taking a (restricted) product of all local factors. We use the combinatorial formulation of the Satake map in the case of $\GL_n$ given by Macdonald using Hall-Littlewood polynomials. 

\begin{thm}[Macdonald]\label{Satake iso}
    The map $\Psi_{K,n}: A_K \rightarrow \mathcal{H}_{n}(K)$
    $$\Psi_{K,n}\left( g(\lambda^{\bullet})^{-1}P_{\lambda^{\bullet}}[\fX_{K}]\right) := \begin{cases}
        \mathbbm{1}_{\pi_{n}^{\lambda^{\bullet}}} & \ell(\lambda^{\bullet}) \leq n \\
        0 & \ell(\lambda^{\bullet}) > n \\
    \end{cases} $$ is an algebra homomorphism and $\Psi_{K,n}$ induces an algebra isomorphism $A_{K,n} \cong \mathcal{H}_{n}(K).$ 
\end{thm}
\begin{proof}
    This follows by applying Macdonald's version of the Satake-isomorphism [(2.7), pg. 296, \cite{Macdonald}] to each local part and taking a product over all local factors.
\end{proof}

\begin{remark}\label{inner product remark}
    The pairing $\langle -,-\rangle$ on $A_{K}$ is only partially compatible with the pairing $\langle -,-\rangle_n$ on $\mathcal{H}_{n}(K).$ The Hall-Littlewood functions $P_{\lambda^{\bullet}}$ of $K$ are mutually orthogonal with respect to $\langle-,-\rangle$ and likewise the double coset indicator functions $\mathbbm{1}_{\pi_{n}^{\lambda^{\bullet}}}$ form a mutually orthogonal basis for $\mathcal{H}_n(K)$ with respect to $\langle -, -\rangle_n.$ However, the norm of $P_{\lambda^{\bullet}}$ is 
    $$\langle P_{\lambda^{\bullet}}, P_{\lambda^{\bullet}} \rangle = b(\lambda^{\bullet})^{-1} =\prod_{\fq \in \bP_K}\prod_{i\geq 1} (1-N(\mathfrak{q})^{-1})^{-1}\cdots (1-N(\mathfrak{q})^{-m_{i}(\lambda^{(\fq)})})^{-1}$$ whereas from [pg. 298, \cite{Macdonald}] 
    $$\langle \mathbbm{1}_{\pi_{n}^{\lambda^{\bullet}}}, \mathbbm{1}_{\pi_{n}^{\lambda^{\bullet}}} \rangle = g(\lambda)^{-1}P_{\lambda^{\bullet}}[\mathfrak{X}_{K}^{(n)}] \Big|_{x_{i}^{(\fq)}=N(\fq)^{i-1}}.$$ These do not typically agree.
\end{remark}

\begin{defn}\label{defn generalization of Macdonald map}
    Let $n \geq 1$. Define $\mathcal{C}_{K,n}:= C(G_n^{+}(K),H_{n}(K))$ be the algebra of $\mathbb{Q}^{\text{ab}}$-valued continuous  $H_{n}(K)$-biinvariant functions on $G_n^{+}(K).$ The product on $\mathcal{C}_{K,n}$ is given by convolution 
    $$f*g(x):= \int_{G_{n}(K)^{+}} f(xy^{-1})g(y) d\mu_{K,n} $$ which is well defined as both f and g have support on $G_n^{+}(K)$ meaning that this integral will always be over a compact region. Note that $\mathcal{H}_{n}(K) \subset \mathcal{C}_{K,n}.$ We define a map 
    $$\widehat{\Psi}_{K,n}: \mathcal{A}_{K} \rightarrow \mathcal{C}_{K,n}$$ extending Macdonald's map $\Psi_{K,n}: A_{K} \rightarrow \mathcal{H}_{n}(K)$ as follows:

    $$\widehat{\Psi}_{K,n} \left( \sum_{\lambda^{\bullet}} c_{\lambda^{\bullet}} ~g(\lambda^{\bullet})^{-1} P_{\lambda^{\bullet}}[\fX_K] \right):= \sum_{\lambda^{\bullet} \in \mathbb{Y}_{K}^{(n)}} c_{\lambda^{\bullet}} \mathbbm{1}_{\pi_{n}^{\lambda^{\bullet}}} \in \mathcal{C}_{K,n}.$$
\end{defn}

\begin{prop}\label{extension of Macdonald map prop}
    The map $\widehat{\Psi}_{K,n}: \mathcal{A}_{K} \rightarrow \mathcal{C}_{K,n}$ is an algebra homomorphism and induces an isomorphism $\mathcal{A}_{K,n} \cong \mathcal{C}_{K,n}.$
\end{prop}
\begin{proof}
    This follows from Theorem \ref{Satake iso}.
\end{proof}

We will write $\varphi_{K,n}:\mathcal{A}_{K,n} \rightarrow \mathcal{C}_{K,n}$ for the above isomorphism $\mathcal{A}_{K,n} \cong \mathcal{C}_{K,n}.$ We now define elements of the algebras $\mathcal{C}_{K,n}$ defined using the Artin symmetric functions.

\begin{defn}
    Given a finite-dimensional complex Galois representation $\rho: \Gal(L/K) \rightarrow \GL(V)$ and any $ n \geq 1$ define $F^{(n)}_{\rho,L/K} \in \mathcal{C}_{K,n}$ by 
    $$F^{(n)}_{\rho,L/K}:= \widehat{\Psi}_{K,n}\left( \sL(\fX_K,\rho;L/K) \right).$$ For $ x\in G_{n}(K)^{+}$ define $\type(x) = \lambda^{\bullet} \in \Par_{K}^{(n)}$ to be the unique double coset index with $ x \in H_{n}(K) \pi_{n}^{\lambda^{\bullet}}H_{n}(K).$
\end{defn}

\begin{cor}
    $$F^{(n)}_{\rho,L/K}(x) = \Tr_{\widetilde{\mathbb{H}}_{\type(x)}(W_{\rho}) }(\sigma_{K})$$
\end{cor}
\begin{proof}
    This follows directly from Theorem \ref{main thm HL expansion} and Definition \ref{defn generalization of Macdonald map}
\end{proof}

\begin{remark}
    Note that that functions $F^{(n)}_{\rho,L/K}(x)$ are stable as $n \rightarrow \infty$ in the sense that if we embed $G_n(K) \rightarrow G_{n+1}(K)$ in the natural way then for $x \in G_n(K)$ 
    $$F^{(n)}_{\rho,L/K}(x) = F^{(n+1)}_{\rho,L/K}(x).$$ As such we may consider the limit 
    $$\widetilde{F}_{\rho,L/K}:= \lim_{n \rightarrow \infty} F^{(n)}_{\rho,L/K}$$ as a continuous function on the topological group 
    $$G_{\infty}(K):= \lim_{n \rightarrow \infty} G_{n}(K) = \GL_{\infty}(\mathbb{A}_K^{f}).$$ The function is biinvariant with respect to the topological group 
    $$H_{\infty}(K):= \lim_{n \rightarrow \infty} H_{n}(K) = \GL_{\infty}(\widehat{\mathcal{O}}_{K}).$$ Since these groups are \textbf{not} locally compact there is no obvious choice of biinvariant measure $\mu_{K,\infty}$ and so the process for defining a spherical Hecke algebra structure on $H_{\infty}(K) \backslash G_{\infty}(K) / H_{\infty}(K)$ is non-obvious. However, we may at least formally define a spherical Hecke algebra structure on $H_{\infty}(K) \backslash G_{\infty}^{+}(K) / H_{\infty}(K)$ where $G_{\infty}^{+}(K) \subset G_{\infty}(K)$ consists of all matrices with entries in $\widehat{\mathcal{O}}_{K}$ by the \textit{inverse limit}
    $$\mathcal{H}_{\infty}(K):= \varprojlim \mathcal{H}_{n}(K).$$ Here we simply transfer the usual inverse limit structure on symmetric polynomials to the spherical Hecke algebra setting via the Satake isomorphism as in Theorem \ref{Satake iso}. However, it would be much more useful and interesting to determine a $H_{\infty}(K)$-biinvariant measure $\mu_{K,\infty}$ on $G_{\infty}^{+}(K)$ which yields a convolution-algebra structure agreeing with the algebra structure on $\mathcal{H}_{\infty}(K).$   
\end{remark}

\begin{example}
    Let $\rho = \mathbbm{1}$ be the trivial representation of the trivial group $\Gal(K/K).$ Then for all $n \geq 1$
    $$F^{(n)}_{\mathbbm{1},K/K} = 1$$ i.e. the constant function on $G_{n}(K)^{+}.$ To see this we have that 
    \begin{align*}
        \sL(\fX_K,\mathbbm{1};K/K)&= \zeta_{K}[\fX_K]\\
        &= \Exp\left[\sum_{\fq \in \mathbb{P}_{K}} \fX^{(\fq)}\right]\\
        &= \prod_{\fq \in \mathbb{P}_K}\Exp[\fX^{(\fq)}]\\
        &= \prod_{\fq \in \mathbb{P}_K} \left( \sum_{\lambda^{(\fq)}} N(\fq)^{-n(\lambda^{(\fq)})} P_{\lambda^{(\fq)}}[\fX^{(\fq)};N(\fq)^{-1}]  \right)\\
        &= \sum_{\lambda^{\bullet} \in \mathbb{Y}_{K}} g(\lambda^{\bullet})^{-1} P_{\lambda^{\bullet}}[\fX_K].\\
    \end{align*}
We could instead, of course, use Theorem \ref{main thm HL expansion} to see this. From here we now find that for any $n\geq 1$, 
\begin{align*}
    F^{(n)}_{\mathbbm{1},K/K}&= \widehat{\Psi}_{K,n}\left(\sL(\fX_K,\mathbbm{1};K/K) \right)\\
    &= \widehat{\Psi}_{K,n}\left( \sum_{\lambda^{\bullet} \in \mathbb{Y}_{K}} g(\lambda^{\bullet})^{-1} P_{\lambda^{\bullet}}[\fX_K] \right)\\
    &=  \sum_{\lambda^{\bullet} \in \mathbb{Y}_{K}^{(n)}} \mathbbm{1}_{\pi_n^{\lambda^{\bullet}}}\\
    &= 1.\\
\end{align*}
    
\end{example}

We require the following definition:

\begin{defn}
    Given a finite extension of number fields $L/K$ and $n \geq 1$ define the homomorphism $R^{(n)}_{L/K}: \mathcal{C}_{L,n} \rightarrow \mathcal{C}_{K,n}$ as 
    $$R^{(n)}_{L/K}:= \varphi_{K,n} \mathcal{N}_{L/K}^{(n)} \varphi_{L,n}^{-1}.$$
\end{defn}

As a simple corollary of the already established properties of the $\sL(\fX_K,\rho;L/K)$, we obtain the following result for the functions $F^{(n)}_{\rho,L/K}.$

\begin{cor}\label{properties of sph hecke elements}
Suppose $E/L/K$ is a tower of finite Galois extensions,  $\rho,\gamma$ are two representations of $\Gal(L/K)$, and $\rho_0$ is a representation of $\Gal(E/L).$ Let $n \geq 1$. The following properties hold:
    \begin{itemize}
        \item $$F^{(n)}_{\rho,L/K}*F^{(n)}_{\gamma,L/K}= F^{(n)}_{\rho \oplus \gamma,L/K}$$
        \item $$F^{(n)}_{\rho,L/K} = F^{(n)}_{\iota_{\Gal(L/K)}^{\Gal(E/K)}(\rho),E/K}$$ 
        \item $$R_{L/K}^{(n)}\left(F^{(n)}_{\rho_0,E/L}\right)= F^{(n)}_{\Ind_{\Gal(E/L)}^{\Gal(E/K)}(\rho_0), E/K}.$$ 
    \end{itemize}
\end{cor}
\begin{proof}
    This follows from Proposition \ref{Artin sym function properties} and Proposition \ref{extension of Macdonald map prop}.
\end{proof}

\subsection{Dirichlet Series}

Now we turn our attention to defining Dirichlet series from the Artin symmetric functions. First, we work formally, avoiding analytic concerns regarding the variable $s$, and then later directly discuss analytic properties. 

\begin{defn}
    The \textbf{\textit{ring of formal Dirichlet}} $\mathcal{D}(s)$ series over the field $\mathbb{Q}^{\text{ab}}$ is the algebra of all formal series of the form 
    $$\sum_{n\in \mathbb{N}} \frac{a_n}{n^s}$$ with $a_n \in \mathbb{Q}^{\text{ab}}$ for all $n \geq 1$ and with product given by 
    $$\left( \sum_{n} \frac{a_n}{n^s} \right) \left( \sum_{m} \frac{b_m}{m^s}\right)= \sum_{N} \left( \sum_{nm = N} a_nb_m \right) N^{-s}.$$
\end{defn} 

We obtain Dirichlet series from the algebra $\mathcal{C}_{K,n}$ using the Mellin transform.

\begin{defn}
    For $n \geq 1$ \textit{\textbf{formal Mellin transform}} is the map 
    $$\zeta: \mathcal{C}_{K,n} \rightarrow \mathcal{D}(s)$$ given by 
    $$\zeta(F):= \int_{G_{n}(K)^{+}} F(x) ||x||^{s} d\mu_{K,n}$$ where 
    $||x|| := \prod_{\fq \in \mathbb{P}_{K}}N(\fq)^{-\nu_{\fq}(\det(x))}$  is the finite adelic norm of $\det(x) \in 
    (\mathbb{A}^{f}_{K})^{*}.$ Here $\nu_{\fq}(u) \in \mathbb{Z}$ is the $\fq$-adic valuation of $u \in (\mathbb{A}^{f}_{K})^{*}.$
\end{defn}

\begin{remark}
    Note that $||x||$ for $x \in G_{n}^{+}(K)$ is always $N^{-1}$ for some whole number $N$ and so if we set $G_{n}(K;N):= \{x \in G_{n}(K) ~~|~~ ||x||= N^{-1} \}$ then
    
    $$\zeta(F) = \int_{G_{n}(K)^{+}} F(x) ||x||^{s} d\mu_{K,n} = \sum_{N \in \mathbb{N}} N^{-s} \int_{G_{n}(K;N)} F(x) d\mu_{K,n}$$ indeed gives a well-defined element of $\mathcal{D}(s).$ Importantly, since the measure of each double coset of $H_{n}(K)\backslash G_{n}^{+}(K)/H_{n}(K)$ is a rational number and $F$ takes values in $\mathbb{Q}^{\text{ab}}$, the values 
    $\int_{G_{n}(K;N)} F(x) d\mu_{K,n}$ are also in $\mathbb{Q}^{\text{ab}}.$
\end{remark}

Applying the Mellin transform to the functions $F^{(n)}_{\rho,L/K}$ gives the following:

\begin{defn}
    Given a finite-dimensional complex representation $\rho$ of $\Gal(L/K)$ and any $ n \geq 1$ define
    $$L^{(n)}(s,\rho;L/K):= \zeta(F^{(n)}_{\rho,L/K}).$$
\end{defn}

\begin{example}
    We will see soon that each $L^{(n)}(s,\rho;L/K):= \zeta(F^{(n)}_{\rho,L/K})$ factors as a product of ordinary Artin L-functions and as such will actually define a meromorphic function. As a preview of this we know that for all $n \geq 1,$ 
    $F^{(n)}_{\mathbbm{1},K/K} = 1$ so by Macdonald [(4.6), pg. 301, \cite{Macdonald}]
    $$L^{(n)}(s,\mathbbm{1};L/K) = \int_{G_{n}(K)^{+}} ||x||^s d \mu_{K,n} = \zeta_{K}(s)\cdots \zeta_{K}(s-n+1)$$ where 
    $\zeta_{K}(s):= \prod_{\fq \in \mathbb{P}_{K}} (1-N(\fq)^{-s})^{-1}$ is the Dedekind zeta function of $K.$
\end{example}

We will now give an expansion for the series $L^{(n)}(s,\rho;L/K)$ using the expansion of the Artin symmetric functions into the arithmetic Hall-Littlewood basis of $A_K.$ However, we will first require the below definitions.

\begin{defn}
For $\lambda^{\bullet} \in \mathbb{Y}_{K}^{(n)}$ define the scalar $\kappa^{(n)}(\lambda^{\bullet})$ as 
    $$\kappa^{(n)}(\lambda^{\bullet}) := g(\lambda^{\bullet})^{-1} \prod_{\fq \in \mathbb{P}_{K}}\left( N(\fq)^{\sum_{i \geq 1} (n-i)\lambda_{i}^{(\fq)}} \frac{v_{n}(N(\fq)^{-1})}{v^{(n)}_{\lambda^{(\fq)}}(N(\fq)^{-1})} \right)$$
    where 
    $$v_{m}(t):= \prod_{i =1}^{m} \left( \frac{1-t^i}{1-t} \right),$$  
    $$v^{(n)}_{\lambda}(t):= \prod_{i \geq 0} v_{m_{i}(\lambda)}(t)$$ and we consider each $\ell(\lambda) \leq n $ partition as a vector in $\mathbb{Z}^n$ by appending sufficiently many zeros on its right side. We also write 
    $$||\lambda^{\bullet}||:= \prod_{\fq \in \mathbb{P}_K} N(\fq)^{|\lambda^{(\fq)}|}.$$
\end{defn}

\begin{remark}\label{kappa simplification remark}
    It follows directly from a result of Macdonald [(2.9), pg. 298, \cite{Macdonald}] that 
    $$\kappa^{(n)}(\lambda^{\bullet}) = \langle \mathbbm{1}_{\pi_{n}^{\lambda^{\bullet}}}, \mathbbm{1}_{\pi_{n}^{\lambda^{\bullet}}} \rangle = \mu_{K,n}(H_{n}(K)\mathbbm{1}_{\pi_{n}^{\lambda^{\bullet}}}H_{n}(K))$$ and further that in fact
    $$\kappa^{(n)}(\lambda^{\bullet}) = g(\lambda)^{-1} P_{\lambda^{\bullet}}[\fX_{K}^{(n)}]\Big|_{x_i^{(\fq)} = N(\fq)^{i-1}}.$$ 
    Also $$\kappa^{(n)}(\lambda^{\bullet}) =  \prod_{\fq \in \mathbb{P}_{K}}\left( N(\fq)^{\sum_{i \geq 1} (n-2i+1)\lambda_{i}^{(\fq)}} \frac{v_{n}(N(\fq)^{-1})}{v^{(n)}_{\lambda^{(\fq)}}(N(\fq)^{-1})} \right)$$ which we can further simplify to 
    $$\kappa^{(n)}(\lambda^{\bullet}) = ||\lambda^{\bullet}||^{n-1}g(\lambda^{\bullet})^{-2} \prod_{\fq \in \mathbb{P}_{K}} \frac{v_{n}(N(\fq)^{-1})}{v^{(n)}_{\lambda^{(\fq)}}(N(\fq)^{-1})}.$$
\end{remark}

Recall that if $\rho: \Gal(L/K) \rightarrow \GL(V)$ is a finite-dimensional complex Galois representation then we write $W_{\rho} = \bigotimes_{\fq \in \mathbb{P}_{K}} V_{\fq}$ where $V_{\fq}$ is the representation of $G_{L/K}$ which acts trivially in each component except for acting via $\rho|_{D(\fq)}$ on $V^{I(\fq)}$ in the $\fq$ component. We write $\sigma_{K} = (\sigma_{\fq})_{\fq \in \mathbb{P}_K} \in G_{L/K}$ where $\sigma_{\fq} \in D(\fq)/I(\fq)$ are arbitrary choices of Frobenius elements.

\begin{thm}\label{shifted product of L functions expansion}
    $$L^{(n)}(s,\rho;L/K) = \prod_{i=1}^{n} L(s-i+1,\rho;L/K) = \sum_{\lambda^{\bullet}\in \mathbb{Y}_{K}^{(n)}} \frac{\kappa^{(n)}(\lambda^{\bullet})}{||\lambda^{\bullet}||^{s}} \Tr_{\widetilde{\mathbb{H}}_{\lambda^{\bullet}}(W_{\rho})}(\sigma_{K})$$
\end{thm}
\begin{proof}
    We will prove this result by computing $L^{(n)}(s,\rho;L/K)$ in two different ways. First, by our original definition (Definition \ref{main def for Artin sym})  
    $$\sL(\fX_K,\rho;L/K)= \prod_{\mathfrak{q}\in \bP_K}\prod_{i \geq 1} \det(\Id_{V^{I(\mathfrak{q})}}- x_i^{(\mathfrak{q})}\rho(\sigma_{\mathfrak{q}})|_{V^{I(\mathfrak{q})}})^{-1}$$ so that 
    $$\sL(\fX_K^{(n)},\rho;L/K)= \prod_{\mathfrak{q}\in \bP_K}\prod_{i =1}^{n} \det(\Id_{V^{I(\mathfrak{q})}}- x_i^{(\mathfrak{q})}\rho(\sigma_{\mathfrak{q}})|_{V^{I(\mathfrak{q})}})^{-1}.$$ Since  $$\mu_{K,n}(H_{n}(K)\mathbbm{1}_{\pi_{n}^{\lambda^{\bullet}}}H_{n}(K)) = g(\lambda^{\bullet})^{-1}P_{\lambda^{\bullet}}[\mathfrak{X}_{K}^{(n)}] \Big|_{x_{i}^{(\fq)}=N(\fq)^{i-1}}$$  from  Remark \ref{inner product remark} and $\varphi_{K,n}(g(\lambda)^{-1}P_{\lambda^{\bullet}}[\mathfrak{X}_{K}^{(n)}]) = \mathbbm{1}_{\pi_n^{\lambda^{\bullet}}}$ for $\lambda^{\bullet} \in \Par_{K}^{(n)}$ it follows that the map 
    $$\tau^{(n)}:= \zeta \circ \varphi_{K,n} \circ \Xi_{K,n}: \mathcal{A}_{K} \rightarrow \mathcal{D}(s)$$ has the effect of truncating and then specializing variables:
    $$\tau^{(n)}(F[\fX_{K}]) = F[\mathfrak{X}_{K}^{(n)}] \Big|_{x_{i}^{(\fq)} = N(\fq)^{i-1-s}}.$$ Therefore, 
\begin{align*}
    L^{(n)}(s,\rho;L/K)&= \tau^{(n)}(\sL(\fX_K,\rho;L/K))\\ 
    &= \sL(\fX_K^{(n)},\rho;L/K) \Big|_{x_{i}^{(\fq)}=N(\fq)^{i-1-s}}\\
    &= \prod_{\mathfrak{q}\in \bP_K}\prod_{i =1}^{n} \det(\Id_{V^{I(\mathfrak{q})}}- x_i^{(\mathfrak{q})}\rho(\sigma_{\mathfrak{q}})|_{V^{I(\mathfrak{q})}})^{-1} \Big|_{x_{i}^{(\fq)}=N(\fq)^{i-1-s}} \\
    &= \prod_{\mathfrak{q}\in \bP_K}\prod_{i =1}^{n} \det(\Id_{V^{I(\mathfrak{q})}}- N(\fq)^{i-1-s}\rho(\sigma_{\mathfrak{q}})|_{V^{I(\mathfrak{q})}})^{-1}\\
    &= \prod_{i =1}^{n}\prod_{\mathfrak{q}\in \bP_K}\det(\Id_{V^{I(\mathfrak{q})}}- N(\fq)^{i-1-s}\rho(\sigma_{\mathfrak{q}})|_{V^{I(\mathfrak{q})}})^{-1}\\
    &= \prod_{i =1}^{n} L(s-i+1,\rho;L/K).\\
\end{align*}
 Next we recall Theorem \ref{main thm HL expansion} :
 $$\sL(\fX_K,\rho;L/K) = \sum_{\lambda^{\bullet} \in \mathbb{Y}_K}  \Tr_{\widetilde{\mathbb{H}}_{\lambda^{\bullet}}(W_{\rho})}(\sigma_{K})  g(\lambda^{\bullet})^{-1} P_{\lambda^{\bullet}}[\mathfrak{X}_K]$$ so by applying $\tau^{(n)}$ to the right side we see 
\begin{align*}
    L^{(n)}(s,\rho;L/K)&=  \sum_{\lambda^{\bullet} \in \mathbb{Y}_K^{(n)}}  \Tr_{\widetilde{\mathbb{H}}_{\lambda^{\bullet}}(W_{\rho})}(\sigma_{K})  g(\lambda^{\bullet})^{-1} \tau^{(n)}\left(P_{\lambda^{\bullet}}[\mathfrak{X}_K]\right)\\
    &= \sum_{\lambda^{\bullet} \in \mathbb{Y}_K^{(n)}}  \Tr_{\widetilde{\mathbb{H}}_{\lambda^{\bullet}}(W_{\rho})}(\sigma_{K}) g(\lambda^{\bullet})^{-1}  \tau^{(n)}\left(P_{\lambda^{\bullet}}[\fX_{K}^{(n)}]\Big|_{x_{i}^{(\fq)} = N(\fq)^{i-1-s}}\right)\\
    &= \sum_{\lambda^{\bullet} \in \mathbb{Y}_K^{(n)}}  \Tr_{\widetilde{\mathbb{H}}_{\lambda^{\bullet}}(W_{\rho})}(\sigma_{K})  g(\lambda^{\bullet})^{-1} P_{\lambda^{\bullet}}[\fX_{K}^{(n)}]\Big|_{x_{i}^{(\fq)} = N(\fq)^{i-1-s}}\\
    &= \sum_{\lambda^{\bullet} \in \mathbb{Y}_K^{(n)}}  \Tr_{\widetilde{\mathbb{H}}_{\lambda^{\bullet}}(W_{\rho})}(\sigma_{K})  ||\lambda^{\bullet}||^{-s} \kappa^{(n)}(\lambda^{\bullet}).\\
\end{align*}
Now comparing our two calculations gives the result.
\end{proof}

\subsection{Stable Product Identity for Artin L-Functions}

In this section we will expand on Theorem \ref{shifted product of L functions expansion} by taking a limit of shifts of the same Artin L-function. We prove in Theorem \ref{main thm 2} that this limit defines an analytic function on the half-plane $\Re(s) > 1$ and give an expansion of the Dirichlet coefficients in terms of the trace values $\Tr_{\widetilde{\mathbb{H}}_{\lambda^{\bullet}}(W_{\rho})}(\sigma_{K}).$ We will first need to prove multiple preliminary results.

\begin{lem}\label{cancelling terms lemma}
For a partition $\lambda$ and $n \geq 0$
    $$\frac{v_{\ell(\lambda)+n}(t)}{v^{(\ell(\lambda)+n)}_{\lambda*0^n}(t)} = \frac{v_{\ell(\lambda)}(t)}{\prod_{i \geq 1} v_{m_{i}(\lambda)}(t)} \prod_{i=1}^{\ell(\lambda)}\left( \frac{1-t^{n+i}}{1-t^{i}} \right)$$
\end{lem}
\begin{proof}
    This follows by directly expanding out the LHS and cancelling out terms.
\end{proof}

\begin{prop}\label{monotone increasing prop}
    Let $0 \leq t < 1$ and $\lambda$ be a  partition. The sequence 
    $\left( \frac{v_{\ell(\lambda)+n}(t)}{v^{(n)}_{\lambda}(t)} \right)_{n \geq 1}$ is  monotone increasing and 
    $$\lim_{n \rightarrow \infty} \frac{v_{\ell(\lambda)+n}(t)}{v^{(n)}_{\lambda}(t)} = \prod_{i \geq 1}(1-t)^{-1}\cdots (1-t^{m_i(\lambda)})^{-1}$$ 
\end{prop}
\begin{proof}
    Lemma \ref{cancelling terms lemma} gives that 
    $$\frac{v_{\ell(\lambda)+n}(t)}{v^{(\ell(\lambda)+n)}_{\lambda*0^n}(t)} = \frac{v_{\ell(\lambda)}(t)}{\prod_{i \geq 1} v_{m_{i}(\lambda)}(t)} \prod_{i=1}^{\ell(\lambda)}\left( \frac{1-t^{n+i}}{1-t^{i}} \right)$$ which after noticing that 
    $$\frac{v_{\ell(\lambda)}(t)}{\prod_{i \geq 1} v_{m_{i}(\lambda)}(t)}\prod_{i=1}^{\ell(\lambda)}\left( \frac{1}{1-t^{i}} \right) = \prod_{i \geq 1}(1-t)^{-1}\cdots (1-t^{m_i(\lambda)})^{-1}$$ allows us to write 
    $$\frac{v_{\ell(\lambda)+n}(t)}{v^{(\ell(\lambda)+n)}_{\lambda*0^n}(t)} = \left( \prod_{i \geq 1}(1-t)^{-1}\cdots (1-t^{m_i(\lambda)})^{-1} \right) \prod_{i=1}^{\ell(\lambda)}\left( 1-t^{n+i} \right).$$ When $0 \leq t < 1$ the sequence 
    $$\prod_{i=1}^{\ell(\lambda)}\left( 1-t^{n+i} \right)$$ is monotonically increasing to $1$ so we obtain the result.
\end{proof}

\begin{defn}
    For $\lambda^{\bullet} \in \mathbb{Y}_K$ define 
    $$\gamma(\lambda^{\bullet}):= \frac{1}{g(\lambda^{\bullet})^{2}b(\lambda^{\bullet})}.$$ 
\end{defn}

We will require the next proposition. 

\begin{prop}
Let $\lambda^{\bullet} \in \mathbb{Y}_{K}.$ The sequence $\left( \frac{\kappa^{(n)}(\lambda^{\bullet})}{||\lambda^{\bullet}||^{n-1}} \right)_{n \geq 1}$ is monotone increasing with
    $$ \lim_{n \rightarrow \infty} \frac{\kappa^{(n)}(\lambda^{\bullet})}{||\lambda^{\bullet}||^{n-1}} =\gamma(\lambda^{\bullet}).$$
\end{prop}
\begin{proof}
    From Remark \ref{kappa simplification remark} we see 
    $$\frac{\kappa^{(n)}(\lambda^{\bullet})}{||\lambda^{\bullet}||^{n-1}} = g(\lambda^{\bullet})^{-2} \prod_{\fq \in \mathbb{P}_{K}} \frac{v_{n}(N(\fq)^{-1})}{v^{(n)}_{\lambda^{(\fq)}}(N(\fq)^{-1})}$$ so we may use Proposition \ref{monotone increasing prop} to see that each sequence 
    $$ \frac{v_{n}(N(\fq)^{-1})}{v^{(n)}_{\lambda^{(\fq)}}(N(\fq)^{-1})}$$ is monotonically increasing to 
    $$\prod_{i \geq 1}(1-N(\fq)^{-1})^{-1}\cdots (1-N(\fq)^{-m_i(\lambda^{(\fq)})})^{-1} = \frac{1}{b_{\lambda^{(\fq)}}(N(\fq)^{-1})}.$$ Since $\lambda^{\bullet} \in \mathbb{Y}_{K}$ there are only finitely many non-trivial terms in the product of sequences 
    $$\prod_{\fq \in \mathbb{P}_{K}} \frac{v_{n}(N(\fq)^{-1})}{v^{(n)}_{\lambda^{(\fq)}}(N(\fq)^{-1})}$$ and thus we conclude that the sequence 
    $$\prod_{\fq \in \mathbb{P}_{K}} \frac{v_{n}(N(\fq)^{-1})}{v^{(n)}_{\lambda^{(\fq)}}(N(\fq)^{-1})} $$ is monotonically increasing and $$ \prod_{\fq \in \mathbb{Y}_K} b_{\lambda^{(\fq)}}(N(\fq)^{-1})^{-1} = b(\lambda^{\bullet})^{-1}$$ as $n \rightarrow \infty.$ 
\end{proof}

\begin{remark}\label{principal specialization}
    Note that $$\gamma(\lambda^{\bullet}) = \frac{1}{g(\lambda^{\bullet})}P_{\lambda^{\bullet}}[\fX_K]|_{x_i^{(\fq)} = N(\fq)^{-(i-1)}}.$$ This means that the linear map $\psi: A_K \rightarrow \mathbb{Q}^{\text{ab}}$ defined by $\psi(P_{\lambda^{\bullet}}):= g(\lambda^{\bullet})\gamma(\lambda^{\bullet})$ is an algebra homomorphism. This is an arithmetic version of the usual \textit{principal specialization} on symmetric functions given by $f[\fX] \rightarrow f\left[\frac{1}{1-t}\right] = f(1,t,t^2,\ldots).$
\end{remark}

The following is the last analytic proposition we will need to prove Theorem \ref{main thm 2}. This result will be necessary for a dominated convergence argument.

\begin{lem}\label{convergence lem}
    For all $ \epsilon > 0,$ 
    $$\sum_{\lambda^{\bullet} \in \mathbb{Y}_{K}} \frac{\gamma(\lambda^{\bullet})}{||\lambda^{\bullet}||^{1+\epsilon}} < \infty.$$ 
\end{lem}
\begin{proof}
By using Lemma \ref{sum of gamma lemma} and some direct calculation:
    \begin{align*}
        \sum_{\lambda^{\bullet} \in \mathbb{Y}_{K}} \frac{\gamma(\lambda^{\bullet})}{||\lambda^{\bullet}||^{1+\epsilon}} &= \sum_{\mathfrak{a} \in \mathcal{J}_K^{+}} \frac{1}{N(\mathfrak{a})^{1+\epsilon}} \sum_{\substack{\lambda^{\bullet} \in \mathbb{Y}_{K} \\ \mathfrak{a} = \prod_{\fq}\fq^{|\lambda^{(\fq)}|} }} \gamma(\lambda^{\bullet})\\
        &= \sum_{\mathfrak{a} \in \mathcal{J}_K^{+}} \frac{1}{N(\mathfrak{a})^{1+\epsilon}} \prod_{\fq^{\alpha}\mid \mid \mathfrak{a}}\left( \sum_{|\lambda| = \alpha} \gamma(\lambda) \right)\\
        &= \sum_{\mathfrak{a} \in \mathcal{J}_K^{+}} \frac{1}{N(\mathfrak{a})^{1+\epsilon}} \prod_{\fq^{\alpha}\mid \mid \mathfrak{a}}\left( \sum_{|\lambda| = \alpha} \frac{N(\fq)^{-2n(\lambda)}}{b_{\lambda}(N(\fq)^{-1})} \right)\\
        &= \sum_{\mathfrak{a} \in \mathcal{J}_K^{+}} \frac{1}{N(\mathfrak{a})^{1+\epsilon}} \prod_{\fq^{\alpha}\mid \mid \mathfrak{a}}(1-N(\fq)^{-1})^{-1}\cdots (1-N(\fq)^{-\alpha})^{-1}\\
        & \leq \sum_{\mathfrak{a} \in \mathcal{J}_K^{+}} \frac{1}{N(\mathfrak{a})^{1+\epsilon}} \prod_{\fq^{\alpha}\mid \mid \mathfrak{a}} 2^{\alpha} \\
        &= \prod_{\fq \in \mathbb{P}_K} \left( 1-\frac{2}{N(\fq)^{1+\epsilon}} \right)^{-1} \\
        &= \prod_{2 \geq N(\fq)^{\epsilon/2}} \left( 1-\frac{2}{N(\fq)^{1+\epsilon}} \right)^{-1} \prod_{2 < N(\fq)^{\epsilon/2}} \left( 1-\frac{2}{N(\fq)^{1+\epsilon}} \right)^{-1}\\
        & \leq \prod_{2 \geq N(\fq)^{\epsilon/2}} \left( 1-\frac{2}{N(\fq)^{1+\epsilon}} \right)^{-1} \prod_{2 < N(\fq)^{\epsilon/2}} \left( 1-\frac{1}{N(\fq)^{1+\frac{\epsilon}{2}}} \right)^{-1} \\
        &= \prod_{2 \geq N(\fq)^{\epsilon/2}} \left( 1-\frac{2}{N(\fq)^{1+\epsilon}} \right)^{-1} \zeta_{K}(1+\frac{\epsilon}{2}) < \infty .\\
    \end{align*}
    
\end{proof}

Now we show a simple combinatorial lemma we will require in Theorem \ref{main thm 2}. This lemma may be inferred directly from [Example 1, pg. 213, \cite{Macdonald}],[Example 1, pg. 243, \cite{Macdonald}]  but we include a more detailed proof for the sake of completeness.

\begin{lem}\label{sum of gamma lemma}
    \begin{equation}\label{reciprocal b sum expansion}
        (1-t)^{-1}\cdots (1-t^m)^{-1} = \sum_{|\lambda| = m} \frac{t^{2n(\lambda)}}{b_{\lambda}(t)}.
    \end{equation}
\end{lem}
\begin{proof}
First we show:
$$P_{\lambda}\left[ \frac{1}{1-t} ;t \right] = \frac{t^{n(\lambda)}}{b_{\lambda}(t)}.$$
Using the Cauchy identity for Hall-Littlewood functions (Lemma \ref{t cauchy identity}),
$$\Exp[(1-t)XY] = \sum_{\lambda} P_{\lambda}[X]Q_{\lambda}[Y],$$ shows that if we specialize  
$$X \rightarrow \frac{1}{1-t} $$ then we see
$$\Exp[Y] = \sum_{\lambda} P_{\lambda}\left[ \frac{1}{1-t} \right]Q_{\lambda}[Y].$$
This implies that for all $m \geq 0$
$$h_{m}[Y] = \sum_{|\lambda| = m} P_{\lambda}\left[ \frac{1}{1-t} \right]b_{\lambda}(t) P_{\lambda}[Y].$$
On the other hand, as Macdonald shows [Example 1, pg. 243, \cite{Macdonald}],
$$h_{m}[Y] = \sum_{|\lambda| = m} t^{n(\lambda)} P_{\lambda}[Y].$$ By comparing coefficients 
$$t^{n(\lambda)} = P_{\lambda}\left[ \frac{1}{1-t} \right]b_{\lambda}(t)$$

Lastly, since $$h_{m}\left[ \frac{1}{1-t} \right] = (1-t)^{-1}\cdots (1-t^m)^{-1}$$ we find
$$(1-t)^{-1}\cdots (1-t^m)^{-1} = h_{m}\left[ \frac{1}{1-t} \right] = \sum_{|\lambda| = m} t^{n(\lambda)} P_{\lambda}\left[ \frac{1}{1-t} \right] = \sum_{|\lambda| = m} t^{n(\lambda)} \frac{t^{n(\lambda)}}{b_{\lambda}(t)} = \sum_{|\lambda|=m} \frac{t^{2n(\lambda)}}{b_{\lambda}(t)}.$$

\end{proof}

The next lemma shows that the trace values 
$\Tr_{\widetilde{\mathbb{H}}_{\lambda}(V;m)}(g)$ are easy to compute when $V$ is $1$-dimensional.

\begin{lem}\label{HL functors for 1-dim}
    Let $\chi$ be a ($1$-dimensional) character of a finite group $G$ and $\lambda$ a partition. Then 
    $$\Tr_{\widetilde{\mathbb{H}}_{\lambda}(\chi;m)}(g) = \chi^{|\lambda|}(g).$$
\end{lem}
\begin{proof}

Using the Lascoux-Sch\"{u}tzenberger cocharge formula \ref{cocharge formula} we know that 
$$\langle s_{(|\lambda|)}, \widetilde{H}_{\lambda} \rangle_{H} = 1.$$ Further, for any partition $\mu$ with $\ell(\mu) > 1$, $\mathbb{S}_{\mu}(\chi) = 0.$ Thus 
$$\widetilde{\mathbb{H}}_{\lambda}(\chi;m) =  \mathbb{S}_{(|\lambda|)}(\chi) \otimes t^{0}.$$ Now
$\mathbb{S}_{(|\lambda|)}(\chi) = \Sym^{|\lambda|}(\chi) = \chi^{\otimes |\lambda|} = \chi^{|\lambda|}$ and therefore, $\widetilde{\mathbb{H}}_{\lambda}(\chi;m)$ is the character of $G$ given by $g \rightarrow \chi^{|\lambda|}(g)$ and so the result follows.
\end{proof}

Now we are finally ready to prove the main result of this section.

\begin{thm}\label{main thm 2}
The following is an equality of analytic functions on the half plane $\Re(s) > 1$:
    $$\prod_{j \geq 0} L(s+j, \rho ;L/K) = \sum_{\lambda^{\bullet} \in \Par_K} \frac{\gamma(\lambda^{\bullet})}{||\lambda^{\bullet}||^{s}} \Tr_{\widetilde{\mathbb{H}}_{\lambda^{\bullet}}(W_{\rho})}(\sigma_{K}).$$
\end{thm}

\begin{proof}
We begin by proving this statement for $\dim(\rho) = 1.$ To that end we will write $\rho = \chi$ for a group character $\chi$ of $\Gal(L/K)$ and let $\mathfrak{b}(\chi) \in \mathcal{J}_{K}^{+}$ be the ideal given as the product over all $\mathfrak{q} \in \mathbb{P}_K$ such that $\chi|_{I(\mathfrak{q})}$ is nontrivial. We will show that $L^{(n)}(s+n-1,\chi;L/K)$ converges uniformly in compact regions to both the LHS and RHS separately. This will, therefore, show that both LHS and RHS agree in the region $\Re(s) > 1.$

For $n \geq 1$ and $ \Re(s) > 1$ set 
$$h^{(n)}(s):= L^{(n)}(s+n-1,\chi;L/K) = \prod_{j =0}^{n-1} L(s+j,\chi;L/K).$$ It is a straightforward Euler product computation to verify that for all $\Re(s) > 1$
$$\log h^{(n)}(s) = \sum_{(\fq^m,\mathfrak{b}(\chi)) = 1} \frac{1}{m} \frac{\chi(\sigma_{\fq})^m}{N(\fq^m)^s} \frac{1-N(\fq)^{-n}}{1-N(\fq)^{-1}}$$ where we sum over all prime power ideals $\fq^m$ in $\mathcal{J}_{K}^{+}$ which are coprime to $\mathfrak{b}(\chi).$ It suffices to show that for all $\epsilon > 0$
$$\log h^{(n)}(s) \rightarrow \sum_{(\fq^m,\mathfrak{b}(\chi)) = 1} \frac{1}{m} \frac{\chi(\sigma_{\fq})^m}{N(\fq^m)^s} \frac{1}{1-N(\fq)^{-1}}$$ uniformly for $\Re(s) \geq 1 + \epsilon.$ Note that since $|\chi(\sigma_{\fq})^m| = 1$ and $\frac{1}{1-N(\fq)^{-1}} \leq 2,$ the series 
$$\sum_{(\fq^m,\mathfrak{b}(\chi)) = 1} \frac{1}{m} \frac{\chi(\sigma_{\fq})^m}{N(\fq^m)^s} \frac{1}{1-N(\fq)^{-1}}$$ converges absolutely for any $\Re(s) > 1$ and so defines an analytic analytic function in that region. Now we may compute directly. Let $\Re(s) \geq 1+ \epsilon:$
\begin{align*}
   & \left| \sum_{(\fq^m,\mathfrak{b}(\chi)) = 1} \frac{1}{m} \frac{\chi(\sigma_{\fq})^m}{N(\fq^m)^s} \frac{1}{1-N(\fq)^{-1}} - \sum_{(\fq^m,\mathfrak{b}(\chi)) = 1} \frac{1}{m} \frac{\chi(\sigma_{\fq})^m}{N(\fq^m)^s} \frac{1-N(\fq)^{-n}}{1-N(\fq)^{-1}}  \right| \\
   & = \left| \sum_{(\fq^m,\mathfrak{b}(\chi)) = 1} \frac{1}{m} \frac{\chi(\sigma_{\fq})^m}{N(\fq^m)^s} \frac{N(\fq)^{-n}}{1-N(\fq)^{-1}}  \right|\\
   & \leq 2 \left| \sum_{(\fq^m,\mathfrak{b}(\chi)) = 1} \frac{1}{m} \frac{N(\fq)^{-n}}{N(\fq^m)^{1+\epsilon}} \right| \\
   & \leq 2^{1-n} \left| \sum_{(\fq^m,\mathfrak{b}(\chi)) = 1} \frac{1}{m} \frac{1}{N(\fq^m)^{1+\epsilon}} \right| \\
   & \leq \log \zeta_{K}(1+\epsilon) 2^{1-n}.\\
\end{align*}
Since $\lim_{n \rightarrow \infty} \log \zeta_{K}(1+\epsilon) 2^{1-n} = 0$ we conclude that as $n \rightarrow \infty $, $\prod_{j =0}^{n-1} L(s+j,\chi;L/K)$ converges uniformly in compact regions to an analytic function on $\Re(s) > 1.$

It suffices to show that for all $\epsilon > 0$ and $\Re(s) \geq 1 + \epsilon$ the sequence of functions given for $n \geq 1$ as
$$\sum_{\lambda^{\bullet} \in \Par_K^{(n)}} \frac{\kappa^{(n)}(\lambda^{\bullet})}{||\lambda^{\bullet}||^{s+n-1}} \Tr_{\widetilde{\mathbb{H}}_{\lambda^{\bullet}}(W_{\chi})}(\sigma_{K})$$ converges uniformly as $ n \rightarrow \infty$ to 
$$\sum_{\lambda^{\bullet} \in \Par_K} \frac{\gamma(\lambda^{\bullet})}{||\lambda^{\bullet}||^{s}} \Tr_{\widetilde{\mathbb{H}}_{\lambda^{\bullet}}(W_{\chi})}(\sigma_{K}) .$$ Now using Lemma \ref{HL functors for 1-dim}, we know for $\rho = \chi$ that if the set of all $\fq$ with $\lambda^{(\fq)} \neq \emptyset $ is coprime with $\mathfrak{b}(\chi)$ then $\Tr_{\widetilde{\mathbb{H}}_{\lambda^{\bullet}}(W_{\chi})}(\sigma_{K}) = \prod_{\fq} \chi(\sigma_{\fq})^{|\lambda^{(\fq)}|}$ and otherwise $\Tr_{\widetilde{\mathbb{H}}_{\lambda^{\bullet}}(W_{\chi})}(\sigma_{K}) = 0.$
Therefore, we can rewrite the sequence in the form 
$$\sum_{\lambda^{\bullet} \in \mathbb{Y}_K} \mathbbm{1}\left(\ell(\lambda^{\bullet}) \leq n \right) \varphi_{\mathfrak{b}(\chi)}(\lambda^{\bullet}) \prod_{\fq} \chi(\sigma_{\fq})^{|\lambda^{(\fq)}|} \frac{\kappa^{(n)}(\lambda^{\bullet})}{||\lambda^{\bullet}||^{n-1}}||\lambda^{\bullet}||^{-s}$$ where for $\mathfrak{a} \in \mathcal{J}_{K}^{+} $, $\varphi_{\mathfrak{a}}(\lambda^{\bullet})$ is either 0 if some $\lambda^{(\fq)} \neq \emptyset $ has $(\mathfrak{a},\fq) \neq 1$ and $1$ otherwise. 
Then for any $\Re(s) \geq 1 + \epsilon$
\begin{align*}
    &\left| \sum_{\lambda^{\bullet} \in \Par_K} \frac{\gamma(\lambda^{\bullet})}{||\lambda^{\bullet}||^{s}} \varphi_{\mathfrak{b}(\chi)}(\lambda^{\bullet}) \prod_{\fq} \chi(\sigma_{\fq})^{|\lambda^{(\fq)}|}  -     \sum_{\lambda^{\bullet} \in \mathbb{Y}_K} \mathbbm{1}\left(\ell(\lambda^{\bullet}) \leq n \right) \varphi_{\mathfrak{b}(\chi)}(\lambda^{\bullet}) \prod_{\fq} \chi(\sigma_{\fq})^{|\lambda^{(\fq)}|} \frac{\kappa^{(n)}(\lambda^{\bullet})}{||\lambda^{\bullet}||^{n-1}}||\lambda^{\bullet}||^{-s}\right| \\
    &= \left| \sum_{\lambda^{\bullet} \in \Par_K} \frac{\varphi_{\mathfrak{b}(\chi)}(\lambda^{\bullet}) \prod_{\fq} \chi(\sigma_{\fq})^{|\lambda^{(\fq)}|}}{||\lambda^{\bullet}||^{s}} \left( \gamma(\lambda^{\bullet}) - \frac{\kappa^{(n)}(\lambda^{\bullet})}{||\lambda^{\bullet}||^{n-1}}\mathbbm{1}\left(\ell(\lambda^{\bullet}) \leq n \right)  \right)  \right| \\
    & \leq \sum_{\lambda^{\bullet} \in \Par_K} \frac{\varphi_{\mathfrak{b}(\chi)}(\lambda^{\bullet})}{||\lambda^{\bullet}||^{1+\epsilon}} \left| \gamma(\lambda^{\bullet}) - \frac{\kappa^{(n)}(\lambda^{\bullet})}{||\lambda^{\bullet}||^{n-1}}\mathbbm{1}\left(\ell(\lambda^{\bullet}) \leq n \right)  \right|\\
    & = \sum_{\lambda^{\bullet} \in \Par_K} \frac{\varphi_{\mathfrak{b}(\chi)}(\lambda^{\bullet})}{||\lambda^{\bullet}||^{1+\epsilon}} \left( \gamma(\lambda^{\bullet}) - \frac{\kappa^{(n)}(\lambda^{\bullet})}{||\lambda^{\bullet}||^{n-1}}\mathbbm{1}\left(\ell(\lambda^{\bullet}) \leq n \right) \right)\\
    &= \sum_{\lambda^{\bullet} \in \Par_K} \frac{\varphi_{\mathfrak{b}(\chi)}(\lambda^{\bullet})}{||\lambda^{\bullet}||^{1+\epsilon}} \gamma(\lambda^{\bullet}) - \sum_{\lambda^{\bullet} \in \Par_K} \frac{\varphi_{\mathfrak{b}(\chi)}(\lambda^{\bullet})}{||\lambda^{\bullet}||^{1+\epsilon}} \frac{\kappa^{(n)}(\lambda^{\bullet})}{||\lambda^{\bullet}||^{n-1}}\mathbbm{1}\left(\ell(\lambda^{\bullet}) \leq n \right) \\
\end{align*}
Here we have used the fact that $\gamma(\lambda^{\bullet}) 
\geq \frac{\kappa^{(n)}(\lambda^{\bullet})}{||\lambda^{\bullet}||^{n-1}}\mathbbm{1}\left(\ell(\lambda^{\bullet}) \leq n \right)$ and that the above series both converge for all $n \geq 1.$

We now argue that the sequence of series for $n \geq 1$ given by 
$$\sum_{\lambda^{\bullet} \in \Par_K} \frac{\varphi_{\mathfrak{b}(\chi)}(\lambda^{\bullet})}{||\lambda^{\bullet}||^{1+\epsilon}}  \frac{\kappa^{(n)}(\lambda^{\bullet})}{||\lambda^{\bullet}||^{n-1}}\mathbbm{1}\left(\ell(\lambda^{\bullet}) \leq n \right) $$ converges to 
$$\sum_{\lambda^{\bullet} \in \Par_K} \frac{\varphi_{\mathfrak{b}(\chi)}(\lambda^{\bullet})}{||\lambda^{\bullet}||^{1+\epsilon}}\gamma(\lambda^{\bullet}).$$ First, we know that the sequence of functions $g_{n,\epsilon}:\mathbb{Y}_{K} \rightarrow \mathbb{C}$ given by 
$$g_{n,\epsilon}(\lambda^{\bullet}):=  \frac{\varphi_{\mathfrak{b}(\chi)}(\lambda^{\bullet})}{||\lambda^{\bullet}||^{1+\epsilon}}  \frac{\kappa^{(n)}(\lambda^{\bullet})}{||\lambda^{\bullet}||^{n-1}}\mathbbm{1}\left(\ell(\lambda^{\bullet}) \leq n \right)$$ converges weakly increasingly point-wise to the function on $\mathbb{Y}_{K}$ given by 
$$\widehat{g}_{\epsilon}(\lambda^{\bullet}):= \frac{\varphi_{\mathfrak{b}(\chi)}(\lambda^{\bullet})}{||\lambda^{\bullet}||^{1+\epsilon}}  \gamma(\lambda^{\bullet}).$$ Further, we know by Lemma \ref{convergence lem} that the series 
$$\sum_{\lambda^{\bullet} \in \Par_K} \frac{\varphi_{\mathfrak{b}(\chi)}(\lambda^{\bullet})}{||\lambda^{\bullet}||^{1+\epsilon}}\gamma(\lambda^{\bullet})$$ converges. Therefore, by the dominated convergence theorem 
$$\lim_{n \rightarrow \infty} \sum_{\lambda^{\bullet} \in \mathbb{Y}_K} g_{n,\epsilon}(\lambda^{\bullet}) = \sum_{\lambda^{\bullet} \in \mathbb{Y}_K} \lim_{n \rightarrow \infty} g_{n,\epsilon}(\lambda^{\bullet})$$ which means 
$$\lim_{n \rightarrow \infty} \sum_{\lambda^{\bullet} \in \Par_K} \frac{\varphi_{\mathfrak{b}(\chi)}(\lambda^{\bullet})}{||\lambda^{\bullet}||^{1+\epsilon}}  \frac{\kappa^{(n)}(\lambda^{\bullet})}{||\lambda^{\bullet}||^{n-1}}\mathbbm{1}\left(\ell(\lambda^{\bullet}) \leq n \right) = \sum_{\lambda^{\bullet} \in \Par_K} \frac{\varphi_{\mathfrak{b}(\chi)}(\lambda^{\bullet})}{||\lambda^{\bullet}||^{1+\epsilon}}\gamma(\lambda^{\bullet}).$$ Thus in the region $\Re(s) \geq 1 + \epsilon$ we have that 
$$\lim_{n \rightarrow \infty} \sum_{\lambda^{\bullet} \in \Par_K^{(n)}} \frac{\kappa^{(n)}(\lambda^{\bullet})}{||\lambda^{\bullet}||^{s+n-1}} \Tr_{\widetilde{\mathbb{H}}_{\lambda^{\bullet}}(W_{\chi})}(\sigma_{K}) =  \sum_{\lambda^{\bullet} \in \Par_K} \frac{\gamma(\lambda^{\bullet})}{||\lambda^{\bullet}||^{s}} \Tr_{\widetilde{\mathbb{H}}_{\lambda^{\bullet}}(W_{\chi})}(\sigma_{K})$$ uniformly. Therefore, for all $\epsilon > 0$ as $n \rightarrow \infty,$ 
$L^{(n)}(s+n-1,\chi;L/K)$ converges uniformly on $\Re(s) \geq 1+\epsilon$ to 
$$\sum_{\lambda^{\bullet} \in \Par_K} \frac{\gamma(\lambda^{\bullet})}{||\lambda^{\bullet}||^{s}} \Tr_{\widetilde{\mathbb{H}}_{\lambda^{\bullet}}(W_{\chi})}(\sigma_{K})$$ as desired.  

Now suppose $\rho$ is arbitrary. Using Brauer's theorem \ref{Brauer's thm} we may write $$L(s,\rho;L/K) = \prod_{i=1}^{r} L(s,\Ind(\chi_i);L/K)^{a_i}$$ where $\chi_{i}:\Gal(L/K_i) \rightarrow \mathbb{C}^{*}$ are $1$-dimensional characters, $L/K_i/K$ are some intermediate extensions, $\Ind(\chi_i)$ is shorthand for $\Ind_{\Gal(L/K_i)}^{\Gal(L/K)}(\chi)$, and $a_i \in \mathbb{Z}.$ For all $n \geq 0$ we have that 
$$\prod_{j \geq 0}^{n} L(s+j,\rho;L/K) = \prod_{i=1}^{r}\left( \prod_{j \geq 0}^{n} L(s+j,\Ind(\chi_i);L/K) \right)^{a_i}.$$ But now we have that 
$$\prod_{j \geq 0}^{n} L(s+j,\Ind(\chi_i);L/K) = \prod_{j \geq 0}^{n} L(s+j,\chi_i;L/K_i)$$ so we may apply our prior computation to obtain that for all $1\leq i \leq r,$
$$\lim_{n \rightarrow \infty} \prod_{j \geq 0}^{n} L(s+j,\Ind(\chi_i);L/K) = \prod_{j \geq 0} L(s+j,\Ind(\chi_i);L/K)$$ converges uniformly in compact subsets of $\Re(s) > 1$ to an analytic function. But then we know that 
$$\lim_{n \rightarrow \infty} \prod_{j \geq 0}^{n} L(s+j,\rho;L/K) = \prod_{j \geq 0} L(s+j,\rho;L/K)$$ must also converge uniformly in compact subsets of $\Re(s) > 1$ to an analytic function. Similarly, we know that each of the series for $1\leq i \leq r$
$$\sum_{\lambda^{\bullet} \in \Par_{K_i}} \frac{\gamma(\lambda^{\bullet})}{||\lambda^{\bullet}||^{s}} \Tr_{\widetilde{\mathbb{H}}_{\lambda^{\bullet}}(W_{\chi_i})}(\sigma_{K_i})$$ defines an analytic function on $\Re(s) > 1$ agreeing with $\prod_{j \geq 0} L(s+j,\chi_i;L/K_i).$ By using Theorem \ref{main thm HL expansion}, Proposition \ref{Artin sym function properties}, and the algebra homomorphism $\psi$ from Remark \ref{principal specialization} we see that for all $1 \leq i \leq r$
$$\sum_{\lambda^{\bullet} \in \Par_{K_i}} \frac{\gamma(\lambda^{\bullet})}{||\lambda^{\bullet}||^{s}} \Tr_{\widetilde{\mathbb{H}}_{\lambda^{\bullet}}(W_{\chi_i})}(\sigma_{K_i}) = \sum_{\lambda^{\bullet} \in \Par_{K}} \frac{\gamma(\lambda^{\bullet})}{||\lambda^{\bullet}||^{s}} \Tr_{\widetilde{\mathbb{H}}_{\lambda^{\bullet}}(W_{\Ind(\chi_i)})}(\sigma_{K}).$$ Now we have that 
$$\prod_{j \geq 0} L(s+j,\rho;L/K) = \prod_{i=1}^{r} \left( \sum_{\lambda^{\bullet} \in \Par_{K}} \frac{\gamma(\lambda^{\bullet})}{||\lambda^{\bullet}||^{s}} \Tr_{\widetilde{\mathbb{H}}_{\lambda^{\bullet}}(W_{\Ind(\chi_i)})}(\sigma_{K}) \right)^{a_i}$$ and so the RHS must yield a Dirichlet series which converges uniformly for $\Re(s) > 1$ to an analytic function. But again applying Theorem \ref{main thm HL expansion} and the algebra homomorphism $\psi$ from Remark \ref{principal specialization} we see that the coefficients must be given as 
$$\sum_{\lambda^{\bullet} \in \Par_K} \frac{\gamma(\lambda^{\bullet})}{||\lambda^{\bullet}||^{s}} \Tr_{\widetilde{\mathbb{H}}_{\lambda^{\bullet}}(W_{\rho})}(\sigma_{K}).$$

\end{proof}

\begin{remark}
    The coefficients 
    $$\gamma(\lambda^{\bullet}) = \prod_{\fq \in \mathbb{P}_{K}} \left( N(\fq)^{-2n(\lambda^{(\fq)})} \prod_{i \geq 1}  (1-N(\fq)^{-1})^{-1}\cdots (1-N(\fq)^{-m_{i}(\lambda^{(\fq)})})^{-1} \right)$$ do \textit{not} depend on the representation $\rho.$ We may rewrite the expression in Theorem \ref{main thm 2} as 
    $$\prod_{j \geq 0} L(s+j, \rho ;L/K) = \sum_{n \in \mathbb{N}} \frac{\widetilde{\vartheta}_{\rho;L/K}(n)}{n^s} $$ where
    $$\widetilde{\vartheta}_{\rho;L/K}(n):= \sum_{\substack{ \lambda^{\bullet} \in \mathbb{Y}_{K}\\||\lambda^{\bullet}|| = n}} \gamma(\lambda^{\bullet}) \Tr_{\widetilde{\mathbb{H}}_{\lambda^{\bullet}}(W_{\rho})}(\sigma_{K})$$ giving the Dirichlet coefficients $\widetilde{\vartheta}_{\rho;L/K}$ of the series $\prod_{j \geq 0} L(s+j, \rho ;L/K).$ It is also easy to obtain the Euler product for $\prod_{j \geq 0} L(s+j, \rho ;L/K):$
    $$\prod_{j \geq 0} L(s+j, \rho ;L/K) =  \prod_{\mathfrak{q}\in \bP_K}\prod_{j \geq 0} \det  \left(\Id_{V^{I(\mathfrak{q})}}- N(\fq)^{-s}\frac{\rho(\sigma_{\mathfrak{q}})|_{V^{I(\mathfrak{q})}}}{N(\fq)^{j}} \right)^{-1}$$ from which we conclude that the coefficients $\widetilde{\vartheta}_{\rho;L/K}$ are multiplicative. 
\end{remark}

We finish this section by looking at the most basic example of Theorem \ref{main thm 2}. Consider the situation where we have the following:
    \begin{itemize}
        \item $L=K=\mathbb{Q}$
        \item $\Gal(L/K) = 1$
        \item $\rho = \mathbbm{1}$
        \item $L(s,\rho;L/K) = \zeta(s)$
        \item $\sigma_{K} = (1,p)_{p \in \mathbb{P}_{\mathbb{Q}}}$
        \item $W_{\rho} = \mathbbm{1}$
    \end{itemize}
    For any $\lambda^{\bullet} \in \mathbb{Y}_{\mathbb{Q}}$ we then have 
    $$\Tr_{\widetilde{\mathbb{H}}_{\lambda^{\bullet}}(W_{\rho})}(\sigma_{K}) = 1.$$
    
    Then Theorem \ref{main thm 2} in this situation reads as 
    $$\prod_{j \geq 0} \zeta(s+j) = \sum_{\lambda^{\bullet} \in \mathbb{Y}_{\mathbb{Q}}} \frac{\gamma(\lambda^{\bullet})}{||\lambda^{\bullet}||^{s}}.$$ Further simplifying gives 
    $$\prod_{j \geq 0} \zeta(s+j) = \sum_{\lambda^{\bullet} \in \mathbb{Y}_{\mathbb{Q}}} \frac{b(\lambda^{\bullet})^{-1} g(\lambda^{\bullet})^{-2}}{||\lambda^{\bullet}||^{s}}.$$
    Using the standard result 
    $$\prod_{j \geq 0}(1-xt^j)^{-1} = \Exp \left[\frac{x}{1-t} \right] = \sum_{ m \geq 0} \frac{x^{m}}{(1-t)\cdots (1-t^m)}$$ and a straightforward computation with the Euler product form of $\prod_{j \geq 0} \zeta(s+j)$ gives 
    $$\prod_{j \geq 0} \zeta(s+j) = \sum_{\substack{m_{p} \geq 0 \\ m_{p} = 0 \\ \text{almost all p}}} \left( \prod_{p}p^{m_p} \right)^{-s} \prod_{p} (1-p^{-1})^{-1}\cdots (1-p^{-m_p})^{-1}.$$ At first glance the two expansions of $\prod_{j \geq 1} \zeta(s+j)$ we have derived appear different. 
    These two expansions agree only if for all sequences $(m_p)_{p}$ with $m_p \geq 0$ and $m_p = 0$ for almost every $p$ we have 
    $$\prod_{p}(1-p^{-1})^{-1}\cdots (1-p^{-m_p})^{-1} = \sum_{\substack{\lambda^{\bullet}\in \mathbb{P}_{\mathbb{Q}}\\ |\lambda^{(p)}| = m_p}} \frac{g(\lambda^{\bullet})^{-2}}{b(\lambda^{\bullet}) }.$$ By using unique factorization this is equivalent to the statement that for any prime $p$ and $m \geq 0$ 
    $$(1-p^{-1})^{-1}\cdots (1-p^{-m})^{-1} = \sum_{|\lambda| = m} \frac{p^{-2n(\lambda)}}{b_{\lambda}(p^{-1})}$$ but this follows from Lemma \ref{sum of gamma lemma} for $t = p^{-1}.$

\begin{remark}
    We may also rewrite the Dirichlet series coefficients of $\prod_{j \geq 0} \zeta(s+j)$ in a more number-theoretic form as 
    $$\prod_{j \geq 0} \zeta(s+j) = \sum_{n \in \mathbb{N}} \frac{\prod_{p^a||n}(1-p^{-1})^{-1}\cdots (1-p^{-a})^{-1}}{n^s}.$$ Similarly, for any Dirichlet character $\chi$ 
    $$\prod_{j \geq 0} L(s+j,\chi) = \sum_{n \in \mathbb{N}} \frac{\chi(n)}{n^s} \prod_{p^a||n}(1-p^{-1})^{-1}\cdots (1-p^{-a})^{-1}$$ The arithmetic function 
    $\vartheta(n):= \prod_{p^a||n}(1-p^{-1})^{-1}\cdots (1-p^{-a})^{-1}$ is multiplicative. In fact, it is a straightforward exercise to show that $\vartheta$ is the \textit{unique} arithmetic function such that 
    \begin{itemize}
        \item $\vartheta(1) = 1$
        \item $\vartheta(n) = \sum_{d|n} \frac{\vartheta(d)}{d}.$
    \end{itemize} 
\end{remark}

\subsection{Analytic Continuation of Infinite Shifted Artin L-Function Product}

In this final section, we will show that the functions $\prod_{j \geq 0} L(s+j, \rho ;L/K)$ extend to meromorphic functions on $\mathbb{C}.$ These functions are non-abelian generalizations of the higher Dirichlet L-functions defined by Momotani \cite{momotani_06}. However, Momotani completes their Dirichlet series by considering the infinite places of the number field K which we will not do in this paper. In general, the issue of extending the results of this paper to account for the infinite places will be dealt with in future work.

\begin{defn}
    For any finite dimensional complex Galois representation $\rho$ and $ \Re(s) > 1$ define 
    $$\widetilde{L}(s,\rho;L/K):= \prod_{j \geq 0} L(s+j,\rho;L/K).$$
\end{defn}

Theorem \ref{main thm 2} guarantees that $\widetilde{L}(s,\rho;L/K)$ defines an analytic function in the region $ \Re(s) > 1$.

\begin{lem}\label{functional equation}
    For $\Re(s) > 1,$
    \begin{equation}\label{functional eq}
        \widetilde{L}(s,\rho;L/K) = L(s,\rho;L/K) \widetilde{L}(s+1,\rho;L/K).
    \end{equation}
    
\end{lem}
\begin{proof}
    \begin{align*}
        \widetilde{L}(s,\rho;L/K) &= \prod_{j \geq 0} L(s+j,\rho;L/K)\\
        &= L(s,\rho;L/K)\prod_{j \geq 1} L(s+j,\rho;L/K)\\
        &= L(s,\rho;L/K)\prod_{j \geq 0} L(s+j+1,\rho;L/K)\\
        &= L(s,\rho;L/K) \widetilde{L}(s+1,\rho;L/K).\\
    \end{align*}
\end{proof}

We recall the following important result about Artin L-functions.

\begin{thm}[Brauer]\label{mero thm for artin L}
    For any $\rho: \Gal(L/K) \rightarrow \GL(V)$ the function defined on $\Re(s) >1$ defined by $L(s,\rho;L/K)$ can extended to a meromorphic function on $\mathbb{C}.$
\end{thm}
\begin{proof}
    See \cite{Cogdell_07}.
\end{proof}

\begin{thm}\label{Analytic Continuation of Shifted Artin L-Function Product}
    The analytic function $\widetilde{L}(s,\rho;L/K)$ defined on $\Re(s) > 1$ may be extended to a meromorphic function on $\mathbb{C}.$
\end{thm}
\begin{proof}
    Define the function $f(s)$ by 
    $$f(s):= \begin{cases}
        \widetilde{L}(s,\rho;L/K) & ; \Re(s) > 1 \\
        L(s,\rho;L/K)\cdots L(s+ \lceil  \frac{n+1}{2}   \rceil,\rho;L/K) \widetilde{L}(s+\lceil  \frac{n+1}{2}   \rceil +1,\rho;L/K)&; \frac{-n-1}{2} < \Re(s) < \frac{-n+1}{2} < \frac{3}{2} \\
    \end{cases}.$$
    It is straightforward to check that $f(s)$ is well-defined using Theorem \ref{mero thm for artin L} since the defined values of $f(s)$ on each vertical strip agree whenever those strips overlap (using the functional equation in Lemma \ref{functional equation}). 
    Therefore, $f(s)$ is meromorphic on $\mathbb{C}.$ But $f(s)$ agrees with $\widetilde{L}(s,\rho;L/K)$ on $\Re(s) > 1$ so $f(s)$ gives a meromorphic continuation of $\widetilde{L}(s,\rho;L/K)$ to $\mathbb{C}.$
\end{proof}

We observe the following:

\begin{cor}
    Assuming the validity of Artin's conjecture \ref{Artin's Conjecture}, if $\rho$ is non-trivial and irreducible, then $\widetilde{L}(s,\rho;L/K)$ is entire. 
\end{cor}
\begin{proof}
    This is immediate using induction on the functional equation in Lemma \ref{functional eq} and the fact that $\widetilde{L}(s,\rho;L/K)$ has no poles in the region $\Re(s) >1$. The only possible poles of $\widetilde{L}(s,\rho;L/K)$ would have to come from $L(s,\rho;L/K)$ but assuming Artin's conjecture there are none.
\end{proof}

Lastly, we will spend the remainder of this section deriving an easy consequence of Theorem \ref{Analytic Continuation of Shifted Artin L-Function Product}. The Dedekind zeta function $\zeta_K$ has a simple pole at $s =1$ with residue given by the \textbf{\textit{analytic class number formula}}: 
    $$\mathfrak{d}_{K}:= \lim_{s \rightarrow 1}(s-1) \zeta_{K}(s) = \frac{2^{r_1}(2\pi)^{r_2}\mathrm{Reg}_{K}h_{K}}{\omega_{K}\sqrt{N(\mathscr{D}_{K})}}$$ where 
    $r_1+2r_2 = \dim_{\mathbb{Q}} K $ is the number of real and complex embeddings of $K,$ $\mathrm{Reg}_{K}$ is the regulator of $K$, $h_{K} = |\mathrm{Cl}_K|$ is the class number of $K,$ $\omega_{K}$ is the number of roots of unity in $K,$ and $\mathscr{D}_{K}$ is the discriminant ideal of $K/\mathbb{Q}.$ For $K = \mathbb{Q}$ we have that $r_1 = 1$, $r_2 = 0,$ $\mathrm{Reg}_{\mathbb{Q}} = 1$, $h_{\mathbb{Q}} =1$, $\omega_{\mathbb{Q}}$, $\mathscr{D}_{\mathbb{Q}} =1$ so $\zeta_{\mathbb{Q}}(s) = \zeta(s)$ has residue $1$ at $s =1.$ 
    Write $\widetilde{\zeta}_K(s):= \prod_{j \geq 0} \zeta_K(s+j).$ Using Theorem \ref{Analytic Continuation of Shifted Artin L-Function Product}, $\widetilde{\zeta}_K(s)$ extends to a meromorphic function on $\mathbb{C}$ with a simple pole at $s = 1.$ The residue at $s=1$ is given by 
    $$\widetilde{\mathfrak{d}}_{K}:= \lim_{s \rightarrow 1}(s-1)\widetilde{\zeta}_K(s) = \mathfrak{d}_{K} \prod_{j \geq 2} \zeta_{K}(j).$$ 

    We can find another interpretation for the value $\widetilde{\mathfrak{d}}_{K}.$ Consider the function $g_{K}(s)$ given by 
    $$g_{K}(s):= \prod_{m \geq 1}\zeta_{K}(ms).$$ It is straightforward to show that $g_{K}(s)$ defines an analytic function in the half plane $\Re(s)>1$ and extends to a meromorphic function in a neighborhood of $1$ (in fact $\Re(s) > 0$). It is also a good exercise to check that the Dirichlet coefficients of $g_{K}(s)$ are given by 
    $$g_{K}(s) = \sum_{n \in \mathbb{N}} \frac{a_{K}(n)}{n^s}$$ where $a_{K}(n)$ is the number of $\mathcal{O}_{K}$ modules with order $n$ up to isomorphism. On one hand we may calculate the residue of $g_{K}(s)$ at $ s= 1$ directly as 
    $$\lim_{s \rightarrow 1} (s-1)g_{K}(s) = \lim_{s \rightarrow \infty} (s-1) \zeta_{K}(s) \prod_{m \geq 2} \zeta_{K}(ms) = \mathfrak{d}_{K} \prod_{m \geq 2} \zeta_{K}(m) = \widetilde{\mathfrak{d}}_K.$$
    By applying the Tauberian theorem, on the other hand, we see that 
    $$\sum_{n \leq x} a_{K}(n) = \widetilde{\mathfrak{d}}_K x + o(x)$$ as $x \rightarrow \infty.$ Equivalently, 
    $$\widetilde{\mathfrak{d}}_K = \lim_{x \rightarrow \infty} \frac{1}{x} \sum_{n \leq x} a_{K}(n).$$ Therefore, the residue of $ \widetilde{\zeta}_{K}(s)$ at $s=1$ is the asymptotic average of the number of finite $\mathcal{O}_K$ modules. For $K= \mathbb{Q},$  $\widetilde{\mathfrak{d}}_{\mathbb{Q}} = 2.294856591 \ldots .$

    By applying the Tauberian theorem directly to $\widetilde{\zeta}_{K}$ we see that 
    $$\widetilde{\mathfrak{d}}_{K} = \lim_{x \rightarrow \infty} \frac{1}{x} \sum_{n \leq x} \vartheta_{K}(n)$$ 
    where 
    $$\vartheta_{K}(n):= \sum_{||\lambda^{\bullet}||=n } \gamma(\lambda^{\bullet}).$$ Therefore, the arithmetic functions $a_K, \vartheta_{K}$ have the same average order but are not equal. This can be seen from considering the Euler factors of both $\widetilde{\zeta}_K(s)$ and $g_{K}(s)$. For instance if $K= \mathbb{Q}$ then the Euler factor of $\widetilde{\zeta}(s)$ at a prime p is
    $$\prod_{j \geq 0}(1-p^{-s-j})^{-1} = \sum_{j \geq 0} \frac{p^{-js}}{(1-p^{-1})\cdots (1-p^{-j})} $$ and for $g_{\mathbb{Q}}(s)$ the Euler factor of p is 
    $$\prod_{m \geq 1}(1-p^{-ms})^{-1} = \sum_{j \geq 0} \#\{ |\lambda| = j \} ~ p^{-js} $$ where $\#\{ |\lambda| = j \}$ is the ordinary partition counting function. After this calculation the fact that $\vartheta$ and $a_{\mathbb{Q}}$ have the same average order at first seems surprising but becomes more reasonable after applying the identity \eqref{reciprocal b sum expansion}:

    $$(1-p^{-1})^{-1}\cdots (1-p^{-j})^{-1} = \sum_{|\lambda| = j} \frac{p^{-2n(\lambda)}}{b_{\lambda}(p^{-1})} = \sum_{|\lambda| = j} \frac{p^{-2n(\lambda)}}{\prod_{i \geq 1}(1-p^{-1})\cdots (1-p^{-m_i(\lambda)})}.$$ However, it would be interesting to perform a detailed analysis of the asymptotic behavior of the sequences $\vartheta_K.$
    
\printbibliography

@article{Artin_23,
    AUTHOR = {Artin, E.},
     TITLE = {\"{U}ber die {Z}etafunktionen gewisser algebraischer Gruppencharakteren},
   JOURNAL = {Math. Ann.},
      YEAR = {1923},
    NUMBER = {89},
     PAGES = {147--156},
}

@article{Artin_30,
    AUTHOR = {Artin, E.},
     TITLE = {Zur {T}heorie der L-{Reihen} mit allgemeinen {G}ruppencharakteren},
   JOURNAL = {Abh. Math. Sem. Hamburg},
      YEAR = {1930},
    NUMBER = {8},
     PAGES = {292--306},
}

@inproceedings{Cogdell_07,
  title={ON ARTIN L-FUNCTIONS},
  author={James W. Cogdell},
  year={2007},
  url={https://api.semanticscholar.org/CorpusID:7574001}
}

@misc{momotani_06,
      title={Higher Selberg zeta functions for congruence subgroups}, 
      author={Momotani, T.},
      year={2006},
      eprint={math/0504073},
      archivePrefix={arXiv},
      primaryClass={math.NT},
      url={https://arxiv.org/abs/math/0504073}, 
}

@article{LS_78,
    AUTHOR = {Lascoux, A. and Sch\"{u}tzenberger, M.P.},
     TITLE = {Sur une conjecture de {H}.{O}. {F}oulkes},
   JOURNAL = {C.R. Acad. Sci. Paris S\'{e}r.},
      YEAR = {1978},
    NUMBER = {286(7)},
     PAGES = {A323--A324},
}

@book {Macdonald,
    AUTHOR = {Macdonald, I. G.},
     TITLE = {Symmetric functions and {H}all polynomials},
    SERIES = {Oxford Classic Texts in the Physical Sciences},
   EDITION = {Second},
      NOTE = {With contribution by A. V. Zelevinsky and a foreword by
              Richard Stanley,
              Reprint of the 2008 paperback edition [ MR1354144]},
 PUBLISHER = {The Clarendon Press, Oxford University Press, New York},
      YEAR = {2015},
     PAGES = {xii+475},
      ISBN = {978-0-19-873912-8},
   MRCLASS = {05E05 (01A75 05-02 20C30 20C33 20K01 33C80 33D80)},
  MRNUMBER = {3443860},
}

@book {Bump,
    AUTHOR = {Bump, D.},
     TITLE = {Hecke {A}lgebras},
    SERIES = {},
   EDITION = {},
      NOTE = {},
 PUBLISHER = {},
      YEAR = {2010},
     PAGES = {},
      ISBN = {},
   MRCLASS = {},
  MRNUMBER = {},
    URL = {http://sporadic.stanford.edu/bump/math263/hecke.pdf}
}

\end{document}